\documentclass[jsl]{asl}
\usepackage{hyperref}
\usepackage{ifthen}

\usepackage{amssymb}
\usepackage{bussproofs}

\usepackage{times}

\DeclareMathOperator{\dom}{dom}
\DeclareMathOperator{\rng}{rng}

\newtheorem{lemma}{Lemma}[section]
\newtheorem{theorem}[lemma]{Theorem}
\newtheorem{corollary}[lemma]{Corollary}

\newtheorem{claim}{Claim}

\theoremstyle{definition}
\newtheorem{definition}[lemma]{Definition}

\newtheorem*{question}{Question}

 \newenvironment{claimproof}{\begin{proof}}{\renewcommand{\qedsymbol}{$\dashv$}\end{proof}}

\mathchardef\mhyphen="2D

\begin{document}
\author{Henry Towsner}
\revauthor{Towsner, Henry}
\title{Partial Impredicativity in Reverse Mathematics}
\thanks{Partially supported by NSF grant DMS-1157580.}
\address {Department of Mathematics, University of Pennsylvania, 209 South 33rd Street, Philadelphia, PA 19104-6395, USA}
\email{htowsner@math.upenn.edu}
\urladdr{www.math.upenn.edu/~htowsner}

\newcommand{\RCA}{\ensuremath{\mathbf{RCA_0}}}
\newcommand{\ACA}{\ensuremath{\mathbf{ACA_0}}}
\newcommand{\ATR}{\ensuremath{\mathbf{ATR_0}}}
\newcommand{\RT}{\ensuremath{\mathbf{RT^2_2}}}
\newcommand{\Pioo}{\ensuremath{\mathbf{\Pi^1_1\mhyphen CA_0}}}
\newcommand{\Pioom}{\ensuremath{\mathbf{\Pi^1_1\mhyphen CA^-_0}}}
\newcommand{\Piti}[1]{\ensuremath{\mathbf{\Pi^1_{#1}\mhyphen TI_0}}}
\newcommand{\SPiti}[1]{\ensuremath{\mathbf{\Pi}_{#1}\mathbf{(\Pi^1_1)\mhyphen TI_0}}}
\newcommand{\ATRi}[1]{\ensuremath{\mathbf{ATR}_{\mathbf{0}}^{#1}}}
\newcommand{\AutATR}{\ensuremath{\mathbf{Aut-ATR_0}}}
\newcommand{\SDC}{\ensuremath{\mathbf{\Sigma^1_1\mhyphen DC_0}}}

\newcommand{\HT}{\ensuremath{\mathbf{HT}}}
\newcommand{\KT}{\ensuremath{\mathbf{KT}}}
\newcommand{\GHT}{\ensuremath{\mathbf{GHT}}}
\newcommand{\NWT}{\ensuremath{\mathbf{NWT}}}
\newcommand{\LPP}{\ensuremath{\mathbf{LPP_0}}}
\newcommand{\MPP}{\ensuremath{\mathbf{MPP_0}}}
\newcommand{\RMPP}[1]{\ensuremath{\Sigma_{#1}\mhyphen \mathrm{MPP}}}
\newcommand{\RLPP}[1]{\ensuremath{\Sigma_{#1}\mhyphen \mathrm{LPP}}}
\newcommand{\RMPPO}[1]{\ensuremath{\mathbf{\Sigma}_{#1}\mathbf{\mhyphen MPP_0}}}
\newcommand{\RLPPO}[1]{\ensuremath{\mathbf{\Sigma}_{#1}\mathbf{\mhyphen LPP_0}}}
\newcommand{\TLPP}{\ensuremath{\mathbf{TLPP_0}}}
\newcommand{\TMPP}{\ensuremath{\mathbf{TMPP_0}}}

\newcommand{\Pica}[1]{\ensuremath{\mathbf{\Pi^1_1-CA^-_0([1])}}}

\newcommand{\field}{\mathrm{field}}

\begin{abstract} 
In reverse mathematics, it is possible to have a curious situation where we know that an implication does not reverse, but appear to have no information on how to weaken the assumption while preserving the conclusion (other than reducing all the way to the tautology of assuming the conclusion).  A main cause of this phenomenon is the proof of a $\Pi^1_2$ sentence from the theory {\Pioo}.  Using methods based on the functional interpretation, we introduce a family of weakenings of {\Pioo} and use them to give new upper bounds for the Nash-Williams Theorem of wqo theory and Menger's Theorem for countable graphs.
\end{abstract}

\maketitle


\section{Introduction}
The strongest of the ``big five'' systems of Reverse Mathematics is the system {\Pioo}, whose defining axiom, $\Pi^1_1$ comprehension, states that
\[\exists X\forall n[n\in X\leftrightarrow\forall Y\phi(n,Y)]\]
where $\phi$ is an arithmetic formula (that is, a formula without set quantifiers).  This axiom is \emph{impredicative}: the set $X$ is defined in terms of a quantifier over all sets, particularly including the set $X$ itself and sets which may be defined in terms of $X$.

It is impossible for a $\Pi^1_2$ sentence to be equivalent to {\Pioo} (see \cite[Corollary 1.10]{marcone:MR1428011} for a proof); this means that any proof of a $\Pi^1_2$ sentence in {\Pioo} can be optimized to go through in some weaker system.  Despite this, {\Pioo} is the best known upper bound for several $\Pi^1_2$ theorems (in particular, the Nash-Williams Theorem\footnote{Actually, the Nash-Williams Theorem is not $\Pi^1_2$, but rather can be deduced in {\ATR} from a $\Pi^1_2$ sentence provable in \Pioo.} of bqo theory \cite{marcone:MR1428011} and Menger's Theorem for countable graphs \cite{MR2899698}; rather than give the definitions necessary to state these theorems here, they are discussed in detail below).

In this paper, we attempt to resolve this situation in a systematic way: using ideas derived from the functional interpretation, we isolate the portion of {\Pioo} actually being used in these proofs, giving a family of weaker systems with $\Pi^1_2$ axioms, and then show that the proofs in {\Pioo} actually go through, essentially unchanged, in these weaker systems.

Rather than being based on the $\Pi^1_1$ comprehension axiom, we base our systems on the equivalent leftmost path principle:
\begin{quote} 
Let $T$ be an ill-founded tree.  Then there is a leftmost path through $T$.
\end{quote}
Our family of weaker systems use the $\Sigma_\alpha$-relative leftmost path principle:
\begin{quote} 
Let $T$ be an ill-founded tree.  Then there is a path $\Lambda$ through $T$ such that no path through $T$ is both $\Sigma_\alpha$ in $T\oplus\Lambda$ and to the left of $\Lambda$.
\end{quote}
We define \RLPPO{\alpha} to be {\RCA} extended by the $\Sigma_\alpha$-relative leftmost path principle \RLPP{\alpha}, and {\TLPP} to be {\RCA} extended by $\Sigma_\alpha$-relative leftmost path principle for every well-ordering $\alpha$.  Note that these formulations are still fundamentally impredicative: the path $\Lambda$ promised to exist is still to the left of paths which might be defined in terms of $\Lambda$ itself.  However the impredicativity is ``partial'' in the sense that we have restricted, in advance, the complexity of the operations which will be might be used to define paths to the left of $\Lambda$.

Our main results can be summarized as:
\begin{theorem}\ 
\begin{enumerate} 
  \item \RLPPO{0} implies {\ATR} (Theorem \ref{rlpp_atr}).
  \item \RLPPO{2} proves Kruskal's Theorem (Theorem \ref{rmpp_kruskal}).
  \item $\mathbf{A\Pi^1_1\mhyphen TI_0}$ (see \cite[Chapter VII.2]{simpson99}) implies \RLPPO{<\omega} (Theorem \ref{spiti_rmpp}).
  \item {\TLPP} proves the Nash-Williams Theorem (Corollary \ref{tmpp_nwt}).
  \item {\TLPP} proves Menger's Theorem for countable graphs (Theorem \ref{tmpp_menger}).
  \item If {\RLPP{\alpha+2}} holds in a model of {\RCA} then there is an $\omega$-model satisfying \SPiti{\alpha} (Theorem \ref{rlpp_omega_spiti}),
  \item In a model satisfying \SPiti{\alpha+2} and $WO(\alpha)$ with $\alpha$ a successor, \RLPP{\alpha} holds (Theorem \ref{rmpp_upper_bound}).
  \end{enumerate}
\end{theorem}

Unlike the old upper bounds, we do not know of any theoretical obstacle to having a reversal of either the Nash-Williams Theorem or Menger's Theorem for countable graphs to \TLPP.  However the best known lower bound remains \ATR.  It is therefore natural to ask:
\begin{question}
Does either the Nash-Williams Theorem or Menger's Theorem for countable graphs imply \TLPP{} over \ATR{}?
\end{question}

We emphasize that this paper does not give novel proofs of any of the mathematical theorems analyzed; our proof of Kruskal's Theorem is unchanged from Nash-Williams' proof \cite{nash_williams:MR0153601}, our proof of the Nash-Williams Theorem is taken from Marcone's work \cite{marcone:MR1428011}, and our proof of Menger's Theorem is the one given by Shafer \cite{MR2899698}.  Our goal is to illustrate that the methods here isolate the portion of {\Pioo} already being used in existing proofs, without requiring changes to the proofs themselves.

We briefly explain the motivation for the relative leftmost path principle.  The leftmost path principle is a $\Pi^1_3$ sentence.  Consider the analogous situation at the arithmetic level, a $\Pi^0_3$ sentence:
\[\sigma=\forall x\exists y\forall z\phi(x,y,z).\]
If we prove a $\Pi^0_2$ sentence $\tau$ using $\sigma$, we do not expect to need the full strength of $\sigma$ in the proof.  The functional interpretation (see \cite{avigad:MR1640329,kohlenbach:MR2445721}) can be used to extract a function $F$ from the proof of $\sigma\rightarrow\tau$ together with a proof of
\[\left[\forall x\exists y' \forall z\leq F(x,y') \phi(x,y',z)\right]\rightarrow\tau.\]
Informally, a proof of $\sigma\rightarrow\tau$ in a reasonable system cannot actually use the fact that the witness $y(x)$ to $\sigma$ is a genuine witness for all $z$; the proof only used the fact that $y(x)$ is a witness for finitely many particular choices of $z$ (where the particular choices may depend on the value of $y(x)$), and therefore it suffices to use an ``approximate witness'' $y'$ which good enough for this particular proof.

The relative leftmost path principle follows a similar justification.  A proof of a $\Pi^1_2$ sentence from the leftmost path principle cannot depend on having an actual leftmost path; instead, given a supposed leftmost path $\Lambda$, the proof must produce some (now countable instead of merely finite) list of paths (again, depending on $\Lambda$), and use the fact that none of these paths are actually to the left of $\Lambda$.  An appropriate form of the relative leftmost path principle then gives us an ``approximate witness'' which is good enough for a particular proof.  (This analogy between set and numeric quantifiers is a bit misleading if taken too seriously; the arguments given in this paper are actually derived from a functional interpretation for quantifiers over ordinals \cite{avigad:MR2583811}.)

We end the introduction with a short discussion of the proof-theoretic strength of {\TLPP}.  We wish to avoid the technicalities of ordinal analysis in this paper, and nothing else in the paper depends on these comments, so we will be somewhat informal.  The theories \SPiti{\alpha} are well-suited to ordinal analysis (for instance, $\Pi_\alpha(\Pi^1_1)$ formulas embed naturally in the framework of ramified set theory used in \cite{pohlers_handbook}; alternatively, an ordinal analysis could be given by a transfinite generalization of the analysis in \cite{rathjen:MR1212407}).  The results described above show that the proof-theoretic ordinal of {\TLPP} is the smallest $\gamma>0$ such that whenever $\alpha<\gamma$, the proof-theoretic ordinal of \SPiti{\alpha} is also less than $\gamma$.

Since already $\Sigma_{<\omega}\mhyphen\mathrm{LPP}$ implies $\mathbf{A\Pi^1_1\mhyphen TI_0}$, whose proof-theoretic ordinal is the Howard-Bachmann ordinal, the consistency strength of {\TLPP} lies somewhere above the Howard-Bachmann ordinal.  Recall that the usual ordinal notation for the Howard-Bachmann ordinal is $\psi\epsilon_{\Omega_1+1}$ (this notation is explained in detail in \cite{pohlers_handbook}, but note that $\epsilon_\alpha$ is the $\alpha$-th $\epsilon$ number, where $\epsilon_0$ is the proof-theoretic ordinal of Peano arithmetic and \ACA).  The theory {\Pioom}, which adds parameter-free $\Pi^1_1$ comprehension to {\ACA}, has the same proof theoretic ordinal.  If one instead adds parameter-free $\Pi^1_1$ comprehension to {\ATR}, one obtains a theory with proof-theoretic ordinal $\psi\Gamma_{\Omega_1+1}$, where $\Gamma_\alpha$ is the $\alpha$-th fixed point of the Veblen function; most importantly, $\Gamma_0$ is the ordinal of {\ATR}.  (The definition of the collapsing function $\psi$ has to be adjusted to accomodate the presence of the Veblen function, so this notation requires some additional work to make precise.)  Inspection of the embedding of \SPiti{\alpha} into the framework of \cite{pohlers_handbook} shows that the proof-theoretic ordinal of {\TLPP} is at most $\psi\Gamma_{\Omega_1+1}$.  In particular, while the consistency strength of {\TLPP} is above the Howard-Bachmann ordinal, it still requires only one level of impredicativity, whereas {\Pioo} requires $\omega$ levels of impredicativity.

The author thanks Stephen Simpson, Reed Solomon, and the anonymous referees for many helpful suggestions.

\section{Notation}
We briefly recall some notation which will be convenient to use throughout this paper.

We fix, throughout this paper, a computable bijective pairing function $(\cdot,\cdot):\mathbb{N}^2\rightarrow\mathbb{N}$.  We routinely view subsets $S$ of $\mathbb{N}$ as subsets of $\mathbb{N}^2$ by equating $S$ with the set of pairs $x,y$ such that $(x,y)\in S$.

\begin{definition} 
If $S\subseteq\mathbb{N}^2$, we write $\field(S)$ for $\{x\mid\exists y\ (x,y)\in S\text{ or }(y,x)\in S\}$.  We often write $xSy$ for $(x,y)\in S$.  We write $S_x$ for $\{y\mid (x,y)\in S\}$.

By a \emph{partial order}, we mean a set $\prec\subseteq\mathbb{N}^2$ such that:
\begin{enumerate} 
  \item If $x\prec y$ and $y\prec z$ then $x\prec z$,
  \item $x\not\prec x$ for any $x$.
\end{enumerate}


We say $\prec$ is a \emph{linear ordering} if for every $x,y\in\field(\prec)$, either $x\prec y$, $x=y$, or $y\prec x$.

When $\prec$ is a partial order, we write $\preceq$ for the reflexive closure of $\prec$.
\end{definition}
We always use $<$ to denote the usual ordering on $\mathbb{N}$.  We often refer to orderings by a name for the field of the ordering, leaving the underlying order implicit.  For instance, we will refer to a linear order $\alpha$, and to the actual relation as $\prec_\alpha$.

\begin{definition} 
A \emph{sequence from $S$} is a function from a (proper or improper) initial segment of $\mathbb{N}$ to $S$.  A \emph{finite sequence} is a sequence whose domain is finite while an \emph{infinite sequence} is a sequence whose domain is $\mathbb{N}$.  For any $n$, we write $\sigma\upharpoonright n$ for $\sigma\upharpoonright[0,n]$.  We write $\langle q_0,\ldots,q_n\rangle$ for the sequence with $\dom(\langle q_0,\ldots,q_n\rangle)=[0,n]$ and $\langle q_0,\ldots,q_n\rangle(i)=q_i$.  When $\sigma$ is a finite sequence, we write $|\sigma|$ for $|\dom(\sigma)|$.

If $\sigma,\tau$ are sequences, we write $\sigma\sqsubseteq\tau$ to indicate that $\dom(\sigma)\subseteq\dom(\tau)$ and for all $i\in\dom(\sigma)$, $\sigma(i)=\tau(i)$.  If $\sigma$ is a finite sequence, we write $\sigma^\frown\tau$ for the concatenation of $\sigma$ and $\tau$: $\dom(\sigma^\frown\tau)=\dom(\sigma)\cup\{|\sigma|+n\mid n\in\dom(\tau)\}$, $(\sigma^\frown\tau)(i)=\sigma(i)$ if $i<|\sigma|$ and $(\sigma^\frown\tau)(i)=\tau(i-|\sigma|)$ if $i\geq|\sigma|$.

If $\prec$ is a partial order on $S$, we extend $\prec$ to sequences from $S$ by setting $\sigma\prec\tau$ if there is an $\upsilon^\frown\langle n\rangle\sqsubset\sigma,\upsilon^\frown\langle m\rangle\sqsubset\tau$ where $n\prec m$.
\end{definition}
We generally use letters $\sigma,\tau$ for finite sequences and $\Lambda$ for infinite sequences.

\begin{definition} 
If $\prec$ is a partial order, we say $\prec$ is \emph{well-founded}, sometimes written $WF(\prec)$, if there is no infinite sequence $\Lambda$ such that $\Lambda(i+1)\prec\Lambda(i)$ for all $i$.  If $\prec$ is both well-founded and linearly ordered, we say $\prec$ is \emph{well-ordered}, written $WO(\prec)$.  We generally assume that $0$ is the least element of $\prec$.

If $\prec$ is not well-founded, $\prec$ is \emph{ill-founded}.
\end{definition}
Any element $\gamma\in\field(\prec_\alpha)=\alpha$ induces a new partial order, the restriction of $\prec_\alpha$ to $\{\delta\in\field(\alpha)\mid\delta\prec_\alpha\gamma\}$.  We sometimes use $\gamma$ to refer to both the (number coding the) element of $\alpha$ and to the partial order given by the set.

\begin{definition} 
A \emph{tree} a set $T$ of finite sequences such that if $\sigma\in T$ and $\tau\sqsubseteq\sigma$ then $\tau\in T$.  A \emph{path} through $T$ is an infinite sequence $\Lambda$ such that for all $n$, $\Lambda\upharpoonright n\in T$.

We say $T$ is \emph{well-founded} if there does not exist a path through $T$.
\end{definition}
Equivalently, $T$ is well-founded iff the restriction of $\sqsupset$ to $T$ is a well-founded partial order.

We make extensive use in this paper of the standard systems of Reverse Mathematics, particularly {\RCA}, {\ACA}, {\ATR}, and {\Pioo}.  \cite{simpson99} is the standard reference.

\begin{definition} 
If $Y$ is a set and $\prec$ is a partial order, for any $j\in\field(\prec)$ we write $(Y)^j=\{(m,i)\in Y\mid i\prec j\}$ and $(Y)_j=\{m\mid (m,j)\in Y\}$.

If $\theta(x,Y,\vec z,\vec Z)$ is a formula with the displayed free variables, we write $H_\theta(\alpha,Y,\vec z,\vec Z)$ for the formula which says that for every $j\in\alpha$, $(Y)_j=\{x\mid \theta(x,(Y)^j,\vec z,\vec Z)\}$.  When $\theta$ is a universal $\Sigma_1$ formula, we just write $H(\alpha,Y, Z)$, omitting the other parameters.
\end{definition}
When we are dealing with an $\omega$-model and $\alpha$ is a computable well-ordering, $H(\alpha,Y,Z)$ just means that $Y=Z^{(\alpha)}$.  Recall that the main axiom of {\ATR} is $\forall \vec z\forall \vec Z\forall\alpha(WO(\alpha)\rightarrow\exists Y H_\theta(\alpha,Y,\vec z,\vec Z))$.

\begin{definition} 
If $\phi$ is a formula, $TI(\alpha,\phi)$ is the formula stating that transfinite induction for $\phi$ holds along $\alpha$:
\[\forall x\in\field(\alpha)\left[\forall y\prec_\alpha x \phi(y)\rightarrow\phi(x)\right]\rightarrow\forall x\in\field(\alpha) \phi(x).\]
\end{definition}
When $W$ is a set, we write $TI(\alpha,W)$ for $TI(\alpha,x\in W)$.

\section{Principles and Claims}
In this section we introduce the main principles we will work with through the rest of this paper.

\begin{definition} 
Let $T$ be a tree and let $\prec$ be a partial order.  A path $\Lambda$ through $T$ is \emph{minimal} (with respect to $\prec$) if there is no path $\Lambda'$ through $T$ such that $\Lambda'\prec\Lambda$.

{\MPP} is the theory consisting of {\RCA} together with the \emph{minimal path principle}:
\begin{quote} 
If $T$ is an ill-founded tree and $\prec$ is well-founded then there exists a minimal path through $T$.
\end{quote}

{\LPP} is {\RCA} together with the restriction of the minimal path principle to the case where $\prec$ is the usual ordering $<$ on the natural numbers.  We call this the \emph{leftmost path principle}.
\end{definition}

We will later show that the minimal and leftmost path principles are equivalent (Theorem \ref{tmpp_is_tlpp}), and in a computable way, so in all the variants we introduce, there will be no difference between the minimal and leftmost versions.

The following is proved in \cite{marcone:MR1428011}:
\begin{theorem}[\RCA]
{\LPP} is equivalent to \Pioo.
\end{theorem}

We introduce a family of restricted forms of {\MPP} and {\LPP}:
\begin{definition} 
For any $n$, \RMPPO{n} is {\RCA} together with the \emph{$\Sigma_n$-relative minimal path principle}:
\begin{quote} 
Whenever $T$ is an ill-founded tree of finite sequences and $\prec$ is a well-founded partial order, there is a path $\Lambda$ through $T$ such that there is no path $\Lambda'$ through $T$ which is $\Sigma_n$ in $T\oplus\Lambda$ such that $\Lambda'\prec\Lambda$.  
\end{quote}

\RLPPO{n} is {\RCA} together with the restriction of the $\Sigma_n$-relative minimal path principle to the case where $\prec$ is $<$.
\end{definition}

When we take {\ATR} to be our base theory, we may extend this definition to higher levels of the jump hierarchy.  We will see later that even \RLPPO{0} implies {\ATR}.

\begin{definition} 
Let $\alpha$ be a well ordering.  If $Z$ is a set, we say $W$ is $\Sigma_\alpha$ in $Z$ if either:
\begin{itemize}
\item $\alpha=\beta+1$ is a successor, $H(\beta,Y,Z)$, and $W$ is computably enumerable in $Y$, or
\item $\alpha$ is a limit, $H(\alpha,Y,Z)$, and $W$ is computable in $Y$.
\end{itemize}
We say $W$ is $\Pi_\alpha$ in $Z$ if the complement of $W$ is $\Sigma_\alpha$ in $Z$.

If $M$ is a model of {\RCA} and $\alpha$ is an ordering in $M$ such that $M\vDash WO(\alpha)$ then we say \RMPP{\alpha}, \emph{$\Sigma_\alpha$-relative minimal path principle}, holds in $M$ if:
\begin{quote} 
Whenever $T$ is an ill-founded tree of finite sequences and $\prec$ is a well-founded partial order, there is a path $\Lambda$ through $T$ such that no set $\Sigma_\alpha$ in $T\oplus\Lambda$ is a path through $T$ to the left of $\Lambda$.
\end{quote}


{\TMPP} is {\RCA} together with the \emph{transfinite minimal path principle}
\begin{quote} 
Whenever $WO(\alpha)$ holds, \RMPP{\alpha} holds.
\end{quote}

\RLPP{\alpha} and {\TLPP} are the restrictions of \RMPP{\alpha} and {\TMPP} respectively to the case where $\prec$ is $<$.
\end{definition}
Note that {\TMPP} and {\TLPP} are axiomitized by $\Pi^1_2$ formulas.

The minimal path principle is inconvenient to analyze, and for that purpose we will introduce some convenient theories of transfinite induction.

\begin{definition} 
A formula is $\Pi^1_1$, which we will also write $\Pi_0(\Pi^1_1)$, if it has the form
\[\forall X \phi(X)\]
where $\phi$ contains no set quantifiers.  A formula is $\Pi_{n+1}(\Pi^1_1)$ if it has the form $\forall x\phi(x)$ where $\phi$ is built from $\Pi_n(\Pi^1_1)$ formulas using propositional connectives ($\wedge,\vee,\rightarrow,$ and $\neg$).

We write \SPiti{n} for {\ACA} together with the scheme:
\[\forall\beta (WF(\beta)\rightarrow TI(\beta,\phi))\]
whenever $\phi$ is a $\Pi_{n}(\Pi^1_1)$ formula.
\end{definition}
Note that a formula is $A\Pi^1_1$ (``arithmetic in $\Pi^1_1$") exactly if the formula is $\Pi_n(\Pi^1_1)$ for some $n$.  In particular, \SPiti{<\omega} is precisely the theory $\mathbf{A\Pi_1^1\mhyphen TI_0}$, whose proof-theoretic strength is precisely the Howard-Bachmann ordinal.  Other theories with the same proof-theoretic strength include $\mathbf{\Pi^1_\infty\mhyphen TI_0}$, the theory extending {\ACA} by full transfinite induction (see \cite[VII.2]{simpson99}) and {\Pioom}, the theory extending {\ACA} by parameter-free $\Pi^1_1$ comprehension (see \cite{pohlers_handbook}).  Despite having the same proof-theoretic strength, $\mathbf{A\Pi_1^1\mhyphen TI_0}$ does not imply either of these other theories.

Since the theory $\mathbf{A\Pi_1^1\mhyphen TI_0}$ is well understood, we introduce a family of transfinite generalizations.  We will show that these transfinite generalizations are intertwined with the properties \RLPP{\alpha}, providing a tool to calibrate the strength of {\TLPP}.
\begin{definition}
Let $M$ be an $\omega$-model of {\RCA} and let $\alpha$ be a well-ordering.  We say $M$ satisfies \SPiti{\alpha} if whenever $\phi(n,X)$ is an arithmetic formula with parameters from $M$, $Z=\{n\mid M\models\forall X\phi(n,X)\}$, and whenever $\prec$ is a relation in the model $M$ such that $M\models WF(\prec)$, and $W$ is $\Pi_\alpha$ in $Z$, $TI(\prec,W)$ holds.
\end{definition}
We do not require that $\alpha$ have any representation in $M$, and the sets $Z$ and $W$ are therefore determined externally to $M$; similarly, whether $TI(\prec,W)$ holds is determined externally to $M$.  On the other hand, the relation $\prec$ need only be well-founded in the sense of $M$.  Consequently ``$M$ satisfies \SPiti{\alpha}'' is not expressed by a formula of second order arithmetic inside $M$.  However we can still ask this question of a given model (taking $\alpha$ to be an actual well-ordering), and when $N$ is a fixed model of {\ATR} such that $N\models WO(\alpha)$ and $M$ is a countably coded $\omega$-model contained in $N$, the statement ``$M$ satisfies \SPiti{\alpha}'' can be expressed in $N$ by a formula of second order arithmetic.  In the latter case, $N$ itself might fail to be an (actual) $\omega$-model, and $M$ is an $\omega$-sub-model of $N$.  Importantly, in either case $Z$ is absolute in $M$, and since $\alpha$ is either actually well-founded, or we are working in a model $N$ such that $N\models WO(\alpha)$, the collection of $\Pi_\alpha$ in $Z$ sets is uniquely determined by $Z$.

\section{Lower Bounds on the Leftmost Path Principle}
We show that even the weakest principle we are considering, \RLPPO{0}, is fairly strong.  We begin by showing that it implies {\ACA}, which will let us use arithmetic comprehension in later proofs, and illustrate our general method.

\begin{theorem} [\RCA]
\RLPP{0} implies \ACA.\footnote{The simplified construction here was pointed out to us by Stephen Simpson.}
\end{theorem}
\begin{proof} 
It suffices to prove $\Sigma_1$ comprehension.  Let $\phi(x,y)$ be a $\Sigma_0$ formula (possibly with parameters).  We say a finite sequence $\sigma$ from $\{0,1\}$ is \emph{valid} if for each $i<|\sigma|$,
\[\sigma(i)=0\Rightarrow\forall j<|\sigma| \neg\phi(i,j).\]
Consider the tree $T$ of valid finite sequences; $T$ is clearly computable from its parameters, and is ill-founded since the function given by $\Lambda_0(i)=1$ for all $i$ is an infinite path through this tree.

Note that if $\Lambda$ is any infinite path through $T$ and $\Lambda(i)=0$ then $\forall y \neg\phi(i,y,S)$: if $\phi(i,m,S)$ then we cannot have any $\sigma\in T$ with $|\sigma|> m$ and $\sigma(i)=0$.

By \RLPP{0}, we may find a path $\Lambda$ so that no infinite path $\Lambda'$ computable from $\Lambda$ is to the left of $\Lambda$.  Suppose $\{i\mid\Lambda(i)>0\}\neq\{i\mid\exists y\phi(i,y,S)\}$.  Since $\Lambda(i)=0$ implies $\forall y\neg\phi(i,y,S)$, it must be that there is some $i$ with $\Lambda(i)>0$ but $\forall y\neg\phi(i,y,S)$.  But then the function
\[\Lambda'(j)=\left\{\begin{array}{ll}
\Lambda(j)&\text{if }j\neq i\\
0&\text{if }j=i\\
\end{array}\right.\]
is also an infinite path through $T$ and easily computable from $\Lambda$.  But $\Lambda'<\Lambda$, contradicting the fact that $\Lambda$ was relatively leftmost.
\end{proof}

\begin{theorem}[\RCA]
\RLPP{0} implies \ATR.
\label{rlpp_atr}
\end{theorem}
\begin{proof} 
It suffices to show transfinite recursion over $\Sigma_1$ formulas.  Suppose $WO(\alpha)$ and let $\theta(x)=\exists y\phi(x,y,Y)$; we will show that $\exists X H_\theta(\alpha,X)$.  Note that since $\phi$ is $\Sigma_0$, for any $i,Y$ such that $\exists y\phi(i,y,Y)$ holds, there is an $m$ such that for any $Y'$ such that $\chi_{Y'}\upharpoonright m=\chi_Y\upharpoonright m$, $\exists y<m\phi(i,y,Y')$.

We will again consider a tree of potential characteristic functions for $Y$.  A finite sequence of natural numbers is \emph{valid} if:
\begin{itemize} 
  \item For any $\gamma\in\field(\alpha)$ and any $i$ such that $\sigma((i,\gamma))=0$, for every $Y$ such that $\chi_Y\upharpoonright\dom(\sigma)=\sigma$ we have $\forall y<|\sigma|\neg\phi(i,y,(Y)^\gamma)$,
  \item If $\sigma((i,\gamma))>1$ then for every $Y$ such that $\chi_Y\upharpoonright\dom(\sigma)=\sigma$ we have $\phi(i,\sigma((i,\gamma))-2,(Y)^\gamma)$,
  \item If there are $j,\delta$ such that $(j,\delta)<(i,\gamma)$, $\delta>_\alpha\gamma$, and $\sigma((j,\delta))\neq 1$ then $\sigma((i,\gamma))\neq 1$.
\end{itemize}
Note that, since $\phi$ is a computable formula, these conditions are arithmetic (indeed, computable), despite the apparent set quantifier.

The idea is that when $\sigma((i,\gamma))=0$, the universal formula should be true, and when $\sigma((i,\gamma))>0$, the existential should be true.  When $\sigma((i,\gamma))=1$, the existential quantifier is ``unjustified'': no witness is required.  When $\sigma((i,\gamma))>1$, however, a witness is required, and $\sigma((i,\gamma))-2$ should be such a witness.

The final condition in the construction of the tree is perhaps the least obvious; the point is that when we set $\sigma((i,\gamma))=0$, we might be depending on the fact that $\sigma((j,\delta))=1$ for some $\delta<_\alpha\gamma$ but $j$ much larger than $i$ so that $(i,\gamma)<(j,\delta)$.  If we wanted to fix a potential characteristic function by setting $\sigma((j,\delta))=0$, we would have to restore $\sigma((i,\gamma))=1$, and since $(i,\gamma)$ appears below $(j,\delta)$, we are no longer moving to the left.  Our solution is to require that once we set $\sigma((i,\gamma))\neq 1$ in a path, we are supposed to be certain about $(j,\delta)$ whenever $\delta<_\alpha\gamma$.  This is enforced by requiring that we actually provide witnesses to existential formulas of all lower ranks.

There are no requirements when $\sigma(x)=1$, so the function $\Lambda_0(x)=1$ for all $x$ is an infinite path through this tree.  By \RLPP{0}, we may find a relatively leftmost path $\Lambda$.  Let $Y=\{i\mid\Lambda(i)>0\}$.  Since {\ACA} satisfies arithmetic transfinite induction, we show by induction on $\gamma\in field(\alpha)$ that $(Y)_\gamma=\{i\mid\exists y\phi(i,y,(Y)^\gamma\}$.  Assume that $H_\theta(\gamma,(Y)^\gamma)$ holds.

Suppose $\exists y\phi(i,y,(Y)^\gamma)$.  Then there is some $m$ such that whenever $\chi_{Y'}\upharpoonright m=\chi_Y\upharpoonright m$, $\exists y<m\phi(i,y,(Y')^\gamma)$.  Then we cannot have $\Lambda((i,\gamma))=0$, so $\Lambda((i,\gamma))>0$ and therefore $i\in(Y)_\gamma$.

Suppose $\forall y\neg\phi(i,y,(Y)^\gamma)$.  If $\Lambda((i,\gamma))>1$, there would be some $m$ such that whenever $\chi_{Y'}\upharpoonright m=\chi_Y\upharpoonright m$, $\phi(i,\Lambda((i,\gamma))-2,(Y')^\gamma)$, contradicting $\forall y\neg\phi(i,y,(Y)^\gamma)$.  If $\Lambda((i,\gamma))=0$, $i\not\in (Y)_\gamma$ as desired.

So suppose $\Lambda((i,\gamma))=1$.  Observe that for $\delta<_\alpha\gamma$, we have $(Y)_\delta=\{i\mid\exists y\phi(i,y,(Y)^\delta\}$.  In particular, if $\Lambda((j,\delta))=1$ and $\delta<_\alpha\gamma$ then there must be some $y$ such that $\phi(i,y,(Y)^\delta)$, and we may therefore computably (in $(Y)^\delta$) find such a $y$; we name this value $y(j,\delta)$.  We define
\[\Lambda'((j,\delta))=\left\{\begin{array}{ll}
y(j,\delta)+2&\text{if }\delta<_\alpha\gamma\text{, }\Lambda((j,\delta))=1\text{, and }(i,\gamma)<(j,\delta)\\
\Lambda((j,\delta))&\text{if }\delta<_\alpha\gamma\text{ and either }\Lambda((j,\delta))\neq 1\text{ or }(j,\delta)<(i,\gamma)\\
1&\text{if }\delta>_\alpha\gamma\\
\Lambda((j,\delta))&\text{if }\delta=\gamma\text{ and }j\neq i\\
0&\text{if }j=i\text{ and }\delta=\gamma\\
\end{array}\right..\]
Note that $\Lambda'<\Lambda$: if $(j,\delta)<(i,\gamma)$ and $\delta\leq_\alpha\gamma$ then $\Lambda'((j,\delta))=\Lambda((j,\delta))$ by definition, while if $\delta>_\alpha\gamma$ then, since $\Lambda$ satisfied the third condition in the definition of the tree, we must have had $\Lambda((j,\delta))=1=\Lambda'((j,\delta))$.

We check that $\Lambda'$ is an infinite path through $T$; let $Y'=\{i\mid\Lambda'(i)>0\}$.  Let $\sigma\sqsubset\Lambda'$ be a finite initial segment.  Suppose $\sigma((j,\delta))=0$; then $\delta\leq_\alpha\gamma$ and either $\Lambda((j,\delta))=0$ or $(i,\gamma)=(j,\delta)$, and since $(Y')^\gamma=(Y)^\gamma$ and $\forall y\neg\phi(j,y,(Y)^\gamma)$, also $\forall y\neg\phi(j,y,(Y')^\gamma)$.

If $\sigma((j,\delta))>1$ then again $\delta\leq_\alpha\gamma$ and $\phi(j,\Lambda'((j,\delta))-2,(Y)^\delta)$, so $\phi(j,\Lambda'((j,\delta))-2,(Y')^\delta)$.

Finally, if there is any $(j',\delta')<(j,\delta)$ with $\delta<_\alpha\delta'$ and $\sigma((j',\delta'))\neq 1$, we have $\delta<_\alpha\delta'\leq_\alpha\gamma$, and therefore $\sigma((j,\delta))\neq 1$.

So $\Lambda'$ is an infinite path computable from $\Lambda$ and to the left of $\Lambda$, which contradicts the choice of $\Lambda$.
\end{proof}

Finally, we give our main lower bound on \RLPPO{\alpha}.  Theorem \ref{spiti_to_rlpp} shows that this bound is almost sharp, leaving a small gap between the amount of transfinite induction we need to obtain \RLPP{\alpha} and the amount we show to be implied by \RLPP{\alpha}. 
\begin{theorem} \label{rlpp_omega_spiti}
Let $N$ be a model of \RCA{} containing an ordering $\alpha$ such that $N\vDash WO(\alpha)$ and $N\vDash\RLPP{\alpha+2}$.  Then
\[N\vDash\text{``there exists a countably coded $\omega$-model of {\ACA} satisfying \SPiti{\alpha}''}.\]
\end{theorem}
The proof gives a bit more, namely that the same claim would hold if $N\vDash\Pi_{\alpha+1}\mhyphen \mathrm{LPP}$, and even $\Pi_{\alpha}\mhyphen \mathrm{LPP}$ if either $\alpha\geq\omega$ or $\alpha$ is odd.
\begin{proof}
Working inside $N$, we will construct a model $M$.  We will view a sequence $\Lambda$ as coding a model $M$ by setting $(i,n)\in M$ iff $\Lambda((i,n))>0$.  Since $M$ will be viewed as a countable coded $\omega$-model, this is saying that $M_i=\{n\mid \Lambda((i,n))>0\}$.

In order to ensure closure under arithmetic comprehension, it will be convenient to have a name for the set $M_i$.  We consider an extension of the language of second order logic by countably many new set constants, $S_1,\ldots$.  (For technical reasons, it will be convenient to assume that this language has existential quantifiers and negation, but no universal quantifier.)  We view $M$ as a model of this extended language by defining $M\vDash n\in S_i$ iff $n\in M_i$.

We will define our tree so that when $\phi$ is an arithmetic formula in this language with a single free variable, the set $M_{(0,\lceil\phi\rceil)}=\{n\mid M\vDash\phi(n)\}$.  This will ensure that we have a model of {\ACA}.  (When we define conditions below, we fix a variable and only discuss $M_{0,\lceil\phi\rceil}$ where no other variables occur free in $\phi$; there are no conditions on other cases.)

The new complication will be ensuring that the model satisfies {\SPiti{\alpha}}.  Suppose $M$ were not a model of {\SPiti{\alpha}}; then there would be some arithmetic $\phi(X,x)$, some $n$, and a sequence $\Upsilon$, $\Pi_\alpha$ in $\{j\mid \forall i\phi(M_i,j)\}$, such that $\Upsilon$ is an infinite descending sequence in $M_n$ (where we view $M_n$ as coding a partial order).  (The key point, of course, will be that $\Upsilon$ is $\Pi_{\alpha+1}$, and so $\Sigma_{\alpha+2}$, in $M$.)  We will ensure that if $M_{(2,n)}$ is non-trivial then it is some descending sequence in $M_n$.  (We will also use $M_{(1,n)}$ to make the coding easier.)

We will handle the dependencies of one set on another in a similar manner to the previous theorem.  For this purpose, we define 
\begin{itemize} 
  \item $lvl((0,\lceil t\in S_i\rceil))=lvl(i)+1$,
  \item $lvl((0,\lceil\phi\rceil))=0$ if $\phi$ is atomic and not of the form $t\in S_i$,
  \item $lvl((0,\lceil\neg\phi\rceil))=lvl((0,\lceil\phi\rceil))+1$,
  \item $lvl((0,\lceil\phi\wedge\psi\rceil))=lvl((0,\lceil\phi\vee\psi\rceil))=\max\{lvl((0,\lceil\phi\rceil)),lvl((0,\lceil\psi\rceil))\}+1$,
  \item $lvl((0,\lceil\exists x\phi\rceil))=lvl((0,\lceil\forall x\phi\rceil))=lvl((0,\lceil\phi[0/x]\rceil))+1$,
  \item $lvl((1,n))=lvl(n)+1$,
  \item $lvl((2,n))=lvl(n)+1$,
  \item $lvl((i,j))=0$ in all other cases.
\end{itemize}


We say a sequence $\sigma$ is \emph{valid} if whenever $\sigma(((i,j),k))$ is defined:
\begin{itemize} 
  \item If $i=0$ and $j=\lceil t\in S_n\rceil$ then $\sigma(((i,j),k))=\sigma((n,k))$,
  \item If $i=0$ and $j=\lceil \phi\rceil$ where $\phi$ is atomic and not of the form $t\in S_i$ then $\sigma(((i,j),k))=1$ if $\phi$ is true and $0$ if $\phi$ is false,
  \item If $i=0$ and $j=\lceil\neg\phi\rceil$ then $\sigma(((i,j),k))=1$ if $\sigma(((i,\lceil\phi\rceil),k))=0$ and $0$ otherwise,
  \item If $i=0$ and $j=\lceil\phi\wedge\psi\rceil$ then $\sigma(((i,j),k))=1$ if both $\sigma(((0,\lceil\phi\rceil),k))>0$ and $\sigma(((0,\lceil\psi\rceil),k))>0$, and $0$ otherwise,
   \item If $i=0$, $j=\lceil\exists x\phi\rceil$, and $\sigma(((i,j),k))=0$ then there is no $u<|\sigma|$ such that $\sigma(((i,\lceil\phi(u)\rceil),k))>0$,
   \item If $i=0$, $j=\lceil\exists x\phi\rceil$, $\sigma(((i,j),k))>1$, and $\sigma(((i,\lceil\phi(\sigma(((i,j),k))-2)\rceil),k))$ is defined then $\sigma(((i,\lceil\phi(\sigma(((i,j),k))-2)\rceil),k))>0$,
	\item If $i=1$, $k>0$, and $\sigma(((i,j),0))=0$ then 
		\begin{enumerate}
		  \item $\sigma(((i,j),k))$ is a sequence $\langle q_0,\ldots,q_k\rangle$,
		  \item Whenever $i<k$ and $\sigma((j,(q_{i+1},q_{i})))$ is defined, $\sigma((j,(q_{i+1},q_{i})))>0$,
		  \item If $k>1$ then $\sigma(((i,j),k-1))\sqsubset\sigma(((i,j),k))$.
		\end{enumerate}
	\item If $i=2$ and $\sigma(((1,j),0))=0$ then $\sigma(((2,j),(k,q))=1$ if $\sigma(((1,j),k))_k=q$ and $0$ otherwise.
\end{itemize}

It is easy to construct an infinite path through this tree (the sequence constantly $1$ will no longer work, because of the conditions for atomic formulas, $\wedge$, and $\neg$, but these cases are easily dealt with).

Let $\Lambda$ be the path given by \RLPP{\alpha+2}.  We show that for all $n$,
\begin{enumerate} 
  \item If $n=(0,\lceil\phi\rceil)$, $M_n=\{i\mid M\vDash\phi(i)\}$,
  \item If $n=(2,m)$ and there is any infinite decreasing sequence in $M_m$ which is $\Sigma_{\alpha+1}$ in $M$ then $M_n$ is such a sequence.
\end{enumerate}
Naturally, we proceed by induction on $lvl(n)$.  The first claim is identical to the argument in the previous theorem.  The second claim is obtained by a similar argument: suppose there is an infinite decreasing sequence $\Upsilon$ in $M_m$ which is $\Sigma_{\alpha+1}$ in $M$.  By construction, if $M_n$ is not such a sequence, we have $\sigma(((1,m),0))\neq 0$, so we obtain a new sequence $\Lambda'$ by setting $\Lambda'(((1,m),0))=0$, $\Lambda'(((1,m),k+1))=\Upsilon\upharpoonright k+2$, and $\Lambda'(((2,m),k))=\Upsilon(k)$, and resetting everything of higher level.  Note that any component which depends on the values at $M_{(1,m)}$ or $M_{(2,m)}$ (for instance, sets defined by formulas containing the constant $S_{(2,m)}$) has a higher index then $m$, and therefore all its indices are greater than $((1,m),0)$.  Since $\Lambda'(((1,m),0))<\Lambda(((1,m),0))$ and $\Lambda'$ is $\Sigma_{\alpha+1}$ in $M$, we obtain a contradiction, so $M_n$ was already an infinite descending sequence in $M_m$, concluding the induction.

This immediately gives that $M$ is a model of {\ACA}.  To see that $M$ satisfies {\SPiti{\alpha}}, observe that if $Y=\{n\mid M\vDash\forall X\phi(X,n)\}$ then $Y$ is $\Pi^0_1$ in $M$, and therefore any set $\Pi_\alpha$ in $Y$ is $\Pi_{\alpha+1}$ in $M$.  In particular, any set defined by a $\Pi_\alpha(\Pi^1_1)$ formula is $\Sigma_{\alpha+2}$ in $M$.  It follows that $M$ satisfies \SPiti{\alpha}.
\end{proof}

Before continuing, we note that there is no difference in strength between the leftmost and minimal path principle.
\begin{theorem} [\RCA]
\label{tmpp_is_tlpp}
\begin{enumerate} 
  \item {\LPP} is equivalent to {\MPP}.
  \item For any $\alpha$, \RLPP{\alpha} is equivalent to \RMPP{\alpha}.
  \item {\TLPP} is equivalent to {\TMPP}.
\end{enumerate}
\end{theorem}
\begin{proof} 
The right to left directions are all trivial.  We prove the left to right direction.

Let $\prec$ be a well-founded partial order.  We define a computable map $\pi$ from ${\field(\prec)}\subseteq\mathbb{N}$ to $\mathbb{N}^{<\omega}$ such that if $x\prec y$ then $\pi(x)<\pi(y)$ (in the lexicographic ordering).  We first define an auxiliary map $\pi'$ inductively.  $\pi'$ will have the property that its image consists only of sequences of even numbers followed by a single odd number.  We define $\pi'$ by the following algorithm: let $y$ be given and suppose $\pi'(x)$ has been defined for all $x<y$.  If there is any $x<y$ such that $y\prec x$, choose $x$ $\prec$-least such that this holds, so $\pi'(x)=\sigma^\frown\langle n\rangle$, and set $\pi'(y)=\sigma^\frown\langle n-1,m\rangle$ where $m$ is the smallest odd number so $\pi'(y)\neq \pi'(z)$ for $z<y$.  If there is no such $x$, set $\pi'(y)=\langle m\rangle$ where $m$ is again the smallest odd number so $\pi'(y)\neq \pi'(z)$ for $z<y$.

\begin{claim} 
If $x\prec y$ then $\pi'(x)<\pi'(y)$.
\end{claim}
\begin{claimproof} 
We proceed by induction on the maximum of $x$ and $y$ with respect to $<$.  

Suppose $y<x$.  We proceed by side induction on $y$ along $\prec$, so assume that whenever $z<x$ and $x\prec z$, $\pi'(x)<\pi'(z)$.  First, assume there is no such $z$, so $y$ is $\prec$-least such that $x\prec y$ and $y<x$.  Then for some $\sigma,n$, $\pi'(y)=\sigma^\frown\langle n\rangle$ while $\pi'(x)=\sigma^\frown\langle n-1,m\rangle$ for some $m$, so certainly $\pi'(x)<\pi'(y)$.  Otherwise, there is some $z<x$ such that $x\prec z\prec y$; then we have $\pi'(z)<\pi'(y)$ by main IH and $\pi'(x)<\pi'(z)$ by side IH, so $\pi'(x)<\pi'(y)$.

Suppose $x<y$.  Suppose there is some $z<y$ such that $y\prec z$, and let $z$ be $\prec$-least such that this is the case.  Then by IH, $\pi'(x)<\pi'(z)=\sigma^\frown\langle n\rangle$ while $\pi'(y)=\sigma^\frown\langle n-1,m\rangle$.  If $\sigma^\frown\langle n-1,m'\rangle\sqsubset\pi'(x)$ for some $m'$ then we have $m'<m$, so $\pi'(x)<\pi'(y)$.  Otherwise, since $\pi'(x)<\sigma^\frown\langle n\rangle$, we must have $\pi'(x)<\sigma^\frown\langle n-1\rangle$, and therefore $\pi'(x)<\pi'(y)$.  If there is no such $z$, $\pi'(y)=\langle m\rangle$ while $\pi'(x)=\langle n\rangle^\frown\tau$ where $n<m$, so again $\pi'(x)<\pi'(y)$.
\end{claimproof}

\begin{claim}
  Suppose $\pi'(x)=\sigma^\frown\langle n\rangle$ and $\pi'(y)\sqsupset\sigma^\frown\langle n-1\rangle$.  Then $y\prec x$.
\end{claim}
\begin{claimproof}
  First, note that by construction $x<y$.  We proceed by induction on $y-x$.  $\pi'(y)$ must have the form $\tau^\frown\langle m-1,m'\rangle$; if $m=n$ then we have $y\prec x$.  Otherwise, there must be some $z$ with $x<z<y$ such that $\pi'(z)=\tau^\frown\langle m\rangle\sqsupset\sigma^\frown\langle n-1\rangle$.  By IH we have $z\prec x$ (since $z-x<y-x$ and $\pi'(z)\sqsupset\sigma^\frown\langle n-1\rangle$) and $y\prec z$ (since $y-z<y-x$ and $\pi'(y)\sqsupset\tau^\frown\langle m-1\rangle$), and since $\prec$ is a partial order, $y\prec x$.
\end{claimproof}

\begin{claim} 
$\{\sigma\mid\exists x\ \sigma\sqsubseteq\pi'(x)\}$ is well-founded.
\end{claim}
\begin{claimproof} 
Suppose not, and let $\sigma_0\sqsubset\sigma_1\sqsubset\cdots$ be an infinite descending sequence.  Since odd numbers are always terminal, each $\sigma_i$ consists only of even numbers.  Each $\sigma_i=\tau_i^\frown\langle n_i-1\rangle$ for some $\tau_i,n_i$, and by the construction of $\pi'$, there must be some $x_i$ such that $\pi'(x_i)=\tau_i^\frown\langle n_i\rangle$.  Observe that $\sigma_i\sqsubset\pi'(x_{i+1})$, and therefore $x_{i+1}\prec x_i$.  Therefore the $x_i$ form an infinite descending sequence through $\prec$, contradicting the fact that $\prec$ is well-founded.
\end{claimproof}

Now we define $\pi(x)=\pi'(x)^\frown\langle x\rangle$.  (The purpose of this suffix is to ensure that the inverse map is computable.)  Given a sequence $\sigma$, define $\pi(\sigma)$ inductively by $\pi(\langle\rangle)=\langle\rangle$ and $\pi(\sigma^\frown\langle n\rangle)=\pi(\sigma)^\frown\pi(n)$.  $\pi$ is clearly injective.

Now let $T$ be an ill-founded tree of finite sequences and define $T'=\{\sigma\mid\exists\tau\in T\ \sigma\sqsubseteq\pi(\tau)\}$.  Since \RLPP{0} implies {\ACA}, $T'$ exists.  If $\Lambda$ is an infinite path through $T$, $\pi(\Lambda)$ is an infinite path through $T'$, so $T'$ is ill-founded.

\begin{claim} 
If $\Lambda'$ is a path through $T'$, there is a unique path $\Lambda$ through $T$ such that $\pi(\Lambda)=\Lambda'$, and $\Lambda'$ is computable from $\Lambda$.
\end{claim}
\begin{claimproof}
Note that, since $\{\sigma\mid\exists x\ \sigma\sqsubseteq\pi(x)\}$ is well-founded, all subsequences of $\Lambda'$ consisting only of even numbers must be finite.  Then we may uniquely decompose $\Lambda'$ into a sequence of blocks
\[\Lambda'=\sigma_0^\frown\langle n_0,m_0\rangle^\frown\sigma_1^\frown\langle n_1,m_1\rangle\cdots\]
where $\sigma_i$ consists only of even numbers and $n_i$ is odd.  Then for each $i$, we must have $\pi(m_i)=\sigma_i^\frown\langle n_i\rangle$, so setting $\Lambda(i)=m_i$, we have $\pi(\Lambda)=\Lambda'$.
\end{claimproof}

Let $\Lambda'$ be a path through $T'$ given by {\LPP} and let $\Lambda$ be the unique path through $T$ such that $\pi(\Lambda)=\Lambda'$.  If $\Lambda^*\prec\Lambda$ then $\pi(\Lambda^*)<\Lambda'$, contradicting the choice of $\Lambda$.  The second and third parts of the claim follow since if $\Lambda^*$ is $\Sigma_\alpha$ in $\Lambda$, $\pi(\Lambda^*)$ is $\Sigma_\alpha$ in $\Lambda'$.
\end{proof}

\subsection{Models of \texorpdfstring{\SDC}{Sigma11-DC}}
In this section we show that, in addition to proving the existence of models of \SPiti{\alpha}, {\TLPP} proves the existence of certain models satisfying $\SDC$.  This will be needed in our proof of Theorem \ref{tmpp_menger}.

We follow almost exactly the notation of \cite[Chapter VIII.4]{simpson99}

\begin{definition} 
We write $\mathcal{O}_+(a,X)$ to mean that $a=(e,i)$ for some $e$ and $i$ and that $e$ is an $X$-recursive index of an $X$-recursive linear ordering $\leq^X_e$ and $i\in\field(<^X_e)$.  If $\mathcal{O}_+(a,X)$ and $\mathcal{O}_+(b,X)$, we write $b<^X_{\mathcal{O}}a$ to mean that $a=(e,i)$, $b=(e,j)$, and $j<^X_e i$.

We write $\mathcal{O}(a,X)$ to mean that $\mathcal{O}_+(a,X)$ and there is no infinite sequence $(a_i)$ such that $a=a_0>^X_{\mathcal{O}}a_1>^X_{\mathcal{O}}>\cdots$.
\end{definition}

%

\begin{theorem}[\TLPP]
If $T$ is an ill-founded tree and $\prec$ is well-founded then there is a countable coded $\omega$-model $M$ such that $T\in M$, $M$ satisfies $\SDC$, and $M$ satisfies that there is a $\prec$-minimal path through $T$.
\label{models_of_dc}
\end{theorem}
\begin{proof} 
We first carry out the proof of Lemma VIII.4.18 of \cite{simpson99}, taking into account that we need to also include a path through $T$ which will become our $\prec$-minimal path.

Let $\mathcal{O}_1(a,T)$ be a $\Sigma_1^1$ formula stating that there is an infinite path $\Lambda$ through $T$ such that:
\begin{enumerate} 
  \item $\mathcal{O}_+(a,T)$,
  \item There is a countably coded $\omega$-model $M$ of {\ACA} such that $T\in M$, $\Lambda\in M$, and $M$ satisfies $\mathcal{O}(a,T)\wedge\exists Y H(a,Y,T\oplus\Lambda)$ and $M$ satisfies that $\Lambda$ is a $\prec$-minimal path through $T$.
\end{enumerate}

If $\mathcal{O}(a,T)$ holds then certainly $\mathcal{O}_1(a,T)$, $a=(e,i)$, and since $<_e^T\upharpoonright i$ is a well-order, there is a $\Lambda$ such that no path computable in a $Y$ satisfying $H(a,Y,T\oplus\Lambda)$ is $\prec$ $\Lambda$.  Since {\TLPP} implies {\ATR}, we have some $Y$ such that $H(a,Y,T\oplus\Lambda)$, and we may take $M$ to be the set of sets Turing reducible to $Y$.

Since $\mathcal{O}_1(a,T)$ is $\Sigma_1^1$, $\mathcal{O}_1(a,T)$ cannot be equivalent to $\mathcal{O}(a,T)$, so there is an $a^*$ such that $\mathcal{O}_1(a^*,T)\wedge\neg\mathcal{O}(a^*,T)$, and therefore an $\omega$-model $M^*$ of {\ACA} such that $T\in M^*$, $\Lambda\in M^*$, $M^*$ satisfies $\mathcal{O}(a^*,T)$, $M^*$ satisfies $\exists Y H(a^*,T\oplus\Lambda,Y)$, and $M^*$ satisfies that $\Lambda$ is a $\prec$-minimal path through $T$.

The proof of Lemma VIII.4.19 of \cite{simpson99} now shows that there is a model $M\subseteq M^*$ of $\SDC$ containing $T$ and $\Lambda$; it follows that $M$ believes $\Lambda$ is a $\prec$-minimal path through $T$.
\end{proof}

\section{Higman's and Kruskal's Theorems}
\begin{definition} 
$Q$ is a \emph{well-quasi-order (wqo)} if $Q$ is a partial order and whenever $\Lambda:\mathbb{N}\rightarrow Q$, there are $i<j$ such that $\Lambda(i)\preceq_Q\Lambda(j)$.

A sequence $\sigma$ from $Q$ is \emph{bad} if there is no $i<j$ such that $\sigma(i)\preceq_Q\sigma(j)$.
\end{definition}
$Q$ is a well-quasi-order iff the tree of bad sequences from $Q$ is well-founded.

\begin{definition} 
If $Q$ is a partial order, $Q^{<\omega}$ is the set of finite sequences from $Q$ and $\prec^{\omega}_Q$ is given by $\sigma\preceq\tau$ iff there is an order-preserving $\pi:[0,|\sigma|-1]\rightarrow[0,|\tau|-1]$ such that $\sigma(i)\preceq\tau(\pi(i))$ for all $i<|\sigma|$.
\end{definition}

Nash-Williams gave the following short proof of Higman's Theorem \cite{nash_williams:MR0153601}:
\begin{theorem}
If $Q$ is a wqo then so is $Q^{<\omega}$.
\end{theorem}
\begin{proof} 
Suppose $Q$ is a wqo but $Q^{<\omega}$ is not.  Define $\sigma\leq\tau$ if $|\sigma|\leq|\tau|$.  Let $\Lambda$ be a leftmost path through the tree of bad sequences from $Q^{<\omega}$.  Clearly $\Lambda(i)\neq\langle\rangle$ for any $i$, since then we would have $\Lambda(i)=\langle\rangle\preceq^{<\omega}_Q\Lambda(i+1)$.  So we may write $\Lambda(i)=\Lambda'(i)^\frown\langle q(i)\rangle$ for all $i$.  Define $c(i,j)=0$ iff $q(i)\preceq_Q q(j)$, and $c(i,j)=1$ otherwise.  By Ramsey's Theorem for Pairs, there is an infinite set $S$ such that $c$ is homogeneous on $S$.

If $c$ were homogeneously $1$, the function $q\upharpoonright S$ would give an infinite sequence in $Q$ contradicting the fact that $Q$ is a wqo.  So $c$ must be homogeneously $0$.  If for any $i<j\in S$, $\Lambda'(i)\preceq^{<\omega}_Q\Lambda'(j)$ then we would have $\Lambda(i)\preceq^{<\omega}_Q\Lambda(j)$ since $q(i)\preceq_Q q(j)$.  This contradicts the construction of $\Lambda$.

Let $\{i_0,i_1,\ldots\}$ be the increasing enumeration of $S$.  Define $\Lambda^*(i)=\Lambda(i)$ if $i<i_0$ and $\Lambda^*(i)=\Lambda'(i_{i-i_0})$ if $i\geq i_0$.  Then for any $i<j$, either $i<j<i_0$, so $\Lambda^*(i)=\Lambda(i)\not\preceq^{<\omega}_Q\Lambda(j)=\Lambda^*(j)$, or $i<i_0\leq j$, in which case $\Lambda^*(i)=\Lambda(i)\not\preceq^{<\omega}_Q\Lambda(i_{j-i_0})\succeq_Q^{<\omega}\Lambda^*(j)$, or $i_0\leq i<j$, in which case $\Lambda^*(i)=\Lambda(i_{i-i_0})\not\preceq^{<\omega}_Q\Lambda(i_{j-i_0})=\Lambda^*(j)$.  So $\Lambda^*$ is an infinite bad sequence and $\Lambda^*<\Lambda$ contradicting the fact that $\Lambda$ is a leftmost path.
\end{proof}
We may observe that $\Lambda^*$ in the proof is $\Sigma_1$, and therefore that this proof goes through without change in \RLPP{1}.

Sch\"utte and Simpson \cite{simpson:MR822617,simpson:MR961012} gave a different proof of Higman's Theorem in {\ACA}.  In particular, their proof shows that if there is an infinite bad sequence $\Lambda$ from $Q^{<\omega}$ then there is an infinite bad sequence $\Lambda'$ from $Q$ such that $\Lambda'$ is $\Sigma_2$ in $\Lambda$.

We now wish to discuss the proof of Kruskal's Theorem; inconveniently, the theorem concerns trees in a slightly different sense than we have been using.  To avoid confusion, we will call these $K$-trees.
\begin{definition} 
A \emph{$K$-tree} is a finite set $T$ together with a partial order $\leq_T$ such that:
\begin{itemize} 
  \item $T$ has a unique root $r\in T$ such that for all $t\in T$, $r\leq_T t$ and if $t\neq r$ then $t\not\leq_T r$, and
  \item If $t\leq_T s$ and $u\leq_T s$ then either $t\leq_Tu$ or $u\leq_T t$.
\end{itemize}

We write $t\wedge_Tu$ for the infimum of $t$ and $u$, so $t\wedge_T u\leq_T t$, $t\wedge_T u\leq_Tu$, and if both $v\leq_T t$ and $v\leq_T u$ then $v\leq_T t\wedge_T u$.

If $Q$ is a quasi-ordering, a \emph{$Q$-labeled $K$-tree} is a pair $(T,f)$ where $T$ is a $K$-tree and $f:T\rightarrow Q$.  We define a quasi-ordering $\prec_K$ on $Q$-labeled $K$-trees by setting $(T,f)\preceq_K(T',f')$ if there is a function $\pi:T\rightarrow T'$ such that for each $t,u\in T$, $\pi(t\wedge_T u)=\pi(t)\wedge_{T'}\pi(u)$ and $f(t)\preceq_Q \pi(f'(t))$.
\end{definition}

\begin{theorem} [\RLPPO{2}]
If $Q$ is a wqo then so are the $Q$-labeled $K$-trees under $\prec_K$.
\label{rmpp_kruskal}
\end{theorem}
\begin{proof} 
Suppose $Q$ is a wqo but the $Q$-labeled $K$-trees are not.  Define $\prec^*_K$ to by setting $(T,f)\prec^*_K(T',f')$ if $|T'|<|T|$.  Then the tree of bad sequences of $Q$-labeled $K$-trees is ill-founded, so let $\Lambda$ be a relatively $\prec^*_K$-minimal bad sequence given by \RLPP{2}.

Given a $Q$-labeled $K$-tree $(T,f)$, let $\mathcal{F}(T,f)$ be the finite set of proper subtrees of $(T,f)$.  If $T$ is a tree, write $r_T$ for the root of $T$ and $\sigma_{T,f}$ for the sequence of immediate successors of $r_T$ (in an arbitrary order).  We may equate $(T,f)$ with the pair $(f(r_T),\sigma_{T,f})\in Q\times\mathcal{F}(T,f)^{<\omega}$.  In particular, if $f(r_T)\preceq_Qf'(r_{T'})$ and $\sigma_{T,f}\preceq^{<\omega}_K\sigma_{T',f'}$ then $(T,f)\preceq_K(T',f')$.

For each $i$, we have $\Lambda(i)=(T_i,f_i)$.  For $i<j$, define $c(i,j)=0$ if $f_i(r_{T_i})\preceq_Q f_j(r_{T_j})$ and $c(i,j)=1$ otherwise.  By Ramsey's Theorem for pairs, we may restrict $\Lambda$ to a subsequence where $c$ is constant, and since $Q$ is a wqo, it must be that $c(i,j)$ is constantly $0$ on this subsequence.  In particular, since $\Lambda(i)\not\preceq_K\Lambda(j)$ when $i<j$, we have $\sigma_{T_i,f_i}\not\preceq^{<\omega}_K\sigma_{T_j,f_j}$.

Let $\mathcal{S}=\bigcup_i\mathcal{F}(\Lambda(i))$.  Then $\Lambda$ gives an infinite bad sequence in $\mathcal{S}^{<\omega}$.  By Higman's Theorem, there is an infinite bad sequence $\Lambda'(i)$ through $\mathcal{S}$.  For each $i$, let $k_i$ be least such that $\Lambda'(i)\in\mathcal{F}(\Lambda(k_i))$.  Let $k=\min_i k_i$ and choose $i$ least such that $k_i=k$.  Define
\[\Lambda^*(j)=\left\{\begin{array}{ll}
\Lambda(j)&\text{if }j<k\\
\Lambda'(j-k+i)&\text{if }k\leq j\\
\end{array}\right.\]
Since $\Lambda^*\upharpoonright k=\Lambda\upharpoonright k$ and $\Lambda^*(k)=\Lambda'(i)\in\mathcal{F}(\Lambda(k))$, we have $\Lambda^*\not\preceq\Lambda$.  To see that $\Lambda^*$ is bad, let $j<j'$ be given; if $j'<k$ then $\Lambda^*(j)=\Lambda(j)\not\preceq\Lambda(j')=\Lambda^*(j')$ and if $k\leq j$ then $\Lambda^*(j)=\Lambda'(j-k+1)\not\preceq\Lambda'(j'-k+1)=\Lambda^*(j')$.  If $j<k\leq j'$ then $\Lambda^*(j)=\Lambda(j)\not\preceq\Lambda(k_{j'-k+1})$ and since $\Lambda^*(j')=\Lambda'(j'-k+1)\in\mathcal{F}(\Lambda(k_{j'-k+1}))$, we must have $\Lambda^*(j)\not\preceq\Lambda^*(j')$.

But then $\Lambda^*$ is an infinite path to the left of $\Lambda$, contradicting the choice of $\Lambda$.

To see that the proof goes through in \RLPPO{2}, we need only observe that we applied Higman's Theorem to a path given by Ramsey's Theorem for Pairs, and since we may choose the path given by Ramsey's Theorem $\mathrm{low}_2$ in $\Lambda$ (see \cite{cholak:MR1825173}), it follows that $\Lambda^*$ can be chosen $\Sigma_2$ in $\Lambda$.
\end{proof}

A complete analysis of the proof-theoretic strength of Kruskal's Theorem was given by Rathjen and Weiermann \cite{rathjen:MR1212407}; \RLPPO{2} is close to (but not exactly) tight, at least with respect to proof-theoretic strength.

\section{The Arithmetic Relative Leftmost Path Principle}
\label{arithmetic_rpp}
In this section we prove the following:
\begin{theorem} 
For every $n>0$, \SPiti{n+2} proves \RLPP{n}.
\end{theorem}

Throughout this section, fix a tree $T$ and a well-ordering $\prec$.  We write $T_\sigma$ for $\{\tau\in T\mid \sigma\sqsubseteq\tau\}$.

All definitions in this section are assumed to be given in {\ACA}.

Before launching into the rather technical proof, we outline the main ideas of the argument.  We will construct a tree $\widehat{\mathcal{T}_n(T)}^+$ with the property that any path through this tree computes a leftmost path through $T$.  Roughly speaking, elements of $\widehat{\mathcal{T}_n(T)}^+$ consist of a distinguished finite sequence in $T$, viewed as a guess at a leftmost path through $T$, together with ``guesses'' at the truth values of finitely many sentences $\Sigma_n$ in the path through $T$, and also together with explicit witnesses showing that certain $\Sigma_n$ formulas fail to define a path further to the left.  An infinite path through this tree will have to correctly predict the value of every $\Sigma_n$ sentence, and produce witnesses showing that no $\Sigma_n$ formula defines a path further to the left; failure to do so will lead to the path being cut off.

If $\widehat{\mathcal{T}_n(T)}^+$ is ill-founded, we will have the desired leftmost path.  If $\widehat{\mathcal{T}_n(T)}^+$ is well-founded, we will have to show that $T$ is well-founded as well; the key idea is that because $\widehat{\mathcal{T}_n(T)}^+$ is well-founded, we may apply transfinite induction along it, though it will take some work to define the right formula to perform transfinite induction with.

We now set about our construction.  $\widehat{\mathcal{T}_n(T)}^+$ will be the last in a tower of trees.

\begin{definition} 
Let $\mathcal{L}$ be the language of first-order arithmetic, including a pairing function $(\cdot,\cdot)$ and the corresponding projections $p_1,p_2$, with a new function symbol $F$ and a new predicate symbol $\hat{T}$.  We define the \emph{rank $n$ formulas} and the \emph{basic rank $n$ formulas} inductively by:
\begin{itemize} 
  \item $F(i)=j$ where $i,j$ are terms is a basic rank $0$ formula,
  \item All other atomic formulas are (non-basic) rank $0$ formulas,
  \item If $\phi$ is a rank $n$ formula then $\exists x\phi$ and $\forall x\phi$ are basic rank $n+1$ formula,
  \item The rank $n$ formulas contain the basic rank $n$ formulas and are closed under $\wedge,\vee,\neg,\rightarrow$.
\end{itemize}

We write $\mathcal{F}_n$ for the collection of basic formulas of rank $n$ and write $rk(\phi)$ for the least $n$ such that $\phi$ is a formula of rank $n$.

When $s$ is a set of $\mathcal{L}$-formulas, we define $\hat s=s\cup\{\hat{T}(n)\mid n\in T\}\cup\{\neg\hat{T}(n)\mid n\not\in T\}$.  We take $\vdash$ to be the usual deduction relation for first-order logic.

\end{definition}

We now define the trees $\mathcal{T}_n(T)$.  An initial segment of $\mathcal{T}_n(T)$ combines a sequence from $T$ with a guess at the values of the formulas $\Sigma_n$ in a path extending this sequence.
\begin{definition} 
For each $n$, define $\mathcal{T}_n(T)$ to consist of those finite sets $s$ of $\mathcal{L}$-formulas such that:
\begin{itemize} 
  \item If $\phi\in s$ then $\phi$ is a closed basic formula of rank $\leq n$,
  \item $\hat s$ is consistent,
  \item If $F(i)=k\in s$ and $i'< i$ then there is a $j'$ such that $F(i')=j'\in s$,
  \item If $F(i)=j\in s$ then the sequence $\langle F(0), \ldots, F(i)\rangle\in T$,
  \item If $\exists x \phi(x)\in s$ then there is some $i$ such that $\hat s\cap\mathcal{F}_{rk(\phi)}\vdash\phi(i)$.
\end{itemize}

We say $s$ \emph{decides} $F(i)=j$ if there is some $j'$ such that $F(i)=j'\in s$; we say $s$ decides $\exists x\phi$ if either $\exists x\phi\in s$ or $\forall x\neg\phi\in s$.

If $i$ is largest such that for some $j$, $F(i)=j\in s$, we write $\sigma_s$ for $\langle F(0),\ldots,F(i)\rangle$.

If $m\leq n$, define $\pi^n_m:\mathcal{T}_n(T)\rightarrow\mathcal{T}_m(T)$ by $\pi^n_m(s)=\{\phi\in s\mid rk(\phi)\leq m\}$.

If $m<n$, $t\in\mathcal{T}_m(T)$, $s\in\mathcal{T}_n(T)$, we write $t\prec^{+1}s$ if there is a formula $\forall x\phi\in s$ with $rk(\phi)=n$ such that $\hat t\vdash\exists x\neg\phi$.  
\end{definition}
Note that the construction of $\mathcal{T}_n(T)$ requires arithmetic comprehension.  (We could probably, at significant additional labor, reduce this to computable comprehension, since we are really only concerned with fairly direct proofs.)

When we write $t\prec^{+1}s$, we are usually interested in the case where $t\supseteq\pi^{n+1}_n(s)$.  In other words, just looking at $\pi^{n+1}_n(s)$, we had not yet found a witness to the formula $\neg\phi$, but $t$ is a way of extending $\pi^{n+1}_n(s)$ so that $\neg\phi$ must be true.  This induces a different element $t'\in\mathcal{T}_n(T)$ with $\pi^{n+1}_n(t')=t\supseteq\pi^{n+1}_n(s)$.  We think of $t'$ as being to the left of $s$ (as the notation $\prec^{+1}$ implies); this means that witnessed existential statements belong to the left of universal statements, and therefore that a leftmost path through $\mathcal{T}_n(T)$ is exactly a path in which we guess $\Sigma_n$ formulas correctly.

\begin{lemma}[\ACA]
If $s\in\mathcal{T}_{n+1}(T)$, $t\in\mathcal{T}_{n}(T)$, $t\supseteq\pi^{n+1}_n(s)$ and $t\not\prec^{+1}s$ then $t\cup s\in\mathcal{T}_{n+1}(T)$.
\end{lemma}
\begin{proof} 
We need only check that $\widehat{t\cup s}$ is consistent.  Suppose not.  Since $t$ is consistent, $\hat t\vdash\neg\phi$ for some $\phi\in s$ of rank $n+1$.  It cannot be that $\phi$ is universal, since then we would have $t\prec^{+1}s$, so $\phi$ must be existential.  But if $\phi$ is existential then $\pi^{n+1}_n(s)\vdash \phi$, and since $t$ is consistent and extends $\pi^{n+1}_n(s)$, we cannot have $\hat t\vdash\neg\phi$.
\end{proof}

\begin{definition} 
For each $n$, we define properties $WF'_n\subseteq\mathcal{T}_n(T)$ and $WF_n\subseteq\mathcal{T}_n(T)$ inductively as follows.

\begin{itemize} 
  \item $WF'_0(t)$ holds if:
	\begin{quote} 
 	  Suppose that for every $\tau\prec\sigma_t$, $T_\tau$ is well-founded.  Then $T_{\sigma_t}$ is well-founded.
	\end{quote}
  \item $WF'_{n+1}(t)$ holds if:
    \begin{quote}
    Suppose that for all $s\supseteq \pi_n^{n+1}(t)$ such that $s\prec^{+1}t$, $WF_n(s)$; then for all $s\supseteq\pi_n^{n+1}(t)$, $WF_n(s)$.
    \end{quote}
  \item $WF_n(t)$ holds if for every $s\supseteq t$ in $\mathcal{T}_n$, $WF'_n(s)$.
\end{itemize}
\end{definition}
We have stated $WF'_0$ and $WF_0$ to emphasize the similarity with $WF'_n$ and $WF_n$, however $WF'_0(t)$ actually immediately implies $WF_0(t)$: if $WF'_0(t)$ holds, $s\supseteq t$, and for every $\tau\prec\sigma_s$, $T_\tau$ is well-founded, then also for every $\tau\prec\sigma_t$, $T_\tau$ is well-founded, and therefore $T_{\sigma_t}$ is well-founded, which implies that $T_{\sigma_s}$ is well-founded.  This means that $WF_0$ is (equivalent to) a Boolean combination of $\Pi_0(\Pi^1_1)$ formulas, and so for each $n>0$, $WF_n$ is (equivalent to) a $\Pi_{n+1}(\Pi^1_1)$ formula.

$WF'_0$ (and therefore $WF_0$) captures the notion of ``not being an initial segment of the leftmost path'': $WF_0(t)$ holds if either the tree above $t$ is well-founded, or if some path to the left is ill-founded.  Thus the only elements failing $WF_0(t)$ are the initial segments of the leftmost path itself.  $WF'_{n+1}$ extends this to the higher order trees; we view $t\in\mathcal{T}_n(T)$ as consisting of two components: $\pi^{n+1}_n(t)$, which is the lower order content which should be addressed by lower order trees, and the remainder.  $WF'_{n+1}$ will be defined so that when $WF'_{n+1}(t)$ fails to hold, it must be that not only does $WF_n(\pi^{n+1}_n(t))$ fail, essentially saying that $\pi^{n+1}_n(t)$ is an initial segment of a leftmost path, but that $t$ is correct about truth values along this leftmost path.  Equivalently, $WF'_{n+1}(t)$ holds if either some $s\prec^{+1}t$ belongs to the leftmost path through $\mathcal{T}_n(T)$, or if no extension of $\pi^{n+1}_n(t)$ which is compatible with $t$ belongs to such a path.

\begin{lemma}[\ACA]
If $WF_n(s)$ and $s\subseteq t$ then $WF_n(t)$.
\end{lemma}
\begin{proof} 
Immediate, since the definition is monotonic.
\end{proof}

\begin{lemma}[\ACA]
If $WF_n(\pi^{n+1}_n(t))$ then $WF_{n+1}(t)$.
\label{wf_increment}
\end{lemma}
\begin{proof} 
Assuming $WF_n(\pi^{n+1}_n(t))$, for every $s\supseteq \pi^{n+1}_n(t)$, $WF_n(s)$.  This implies $WF_{n+1}(t)$.
\end{proof}

\begin{lemma}[\SPiti{1}]
Let $\phi$ be a basic rank $0$ formula, let $s\in\mathcal{T}_0(T)$, and suppose that for every $t\supseteq s$ such that $t$ decides $\phi$, $WF_0(t)$.  Then $WF_0(s)$.
\label{decides_rank_0}
\end{lemma}
\begin{proof} 
$\phi$ has the form $F(i)=j$ for some $j$.  By main induction on $r$, we show that
\begin{quote} 
Whenever $t\supseteq s$ with $|\sigma_{t}|=i+1-r$, $WF_0(t)$.
\end{quote}
If $r=0$, any such $t$ decides $F(i)=j$, and therefore by assumption, $WF_0(t)$.

Suppose the claim holds for $r$ and let $t\supseteq s$ be given with $|\sigma_t|=i+1-(r+1)=i-r$.  If there is a $\tau\prec\sigma_t$ such that $T_\tau$ is ill-founded then we immediately have $WF_0(t)$.  So assume that for every $\tau\prec\sigma_t$, $T_\tau$ is well-founded.  For each $k$, let $t_k=t\cup\{F(|\sigma_{t}|)=k\}$.  By side induction on $k$ along $\prec$, we will show that $T_{\sigma_{t_k}}$ is well-founded.  Suppose that for all $k'\prec k$ with $t_{k'}\in\mathcal{T}_0(T)$, $T_{\sigma_{t_{k'}}}$ is well-founded.  Since $|\sigma_{t_k}|=|\sigma_t|+1=i+1-r$, we have $WF_0(t_k)$.  If $\tau\prec\sigma_{t_k}$ and $\tau\in T$, we either have $\tau\prec\sigma_t$, in which case we have assumed $T_\tau$ is well-founded, or $\tau=\sigma_t^\frown\langle k'\rangle=\sigma_{t_{k'}}$ for some $k'\prec k$, in which case we have that $T_\tau$ is well-founded by side IH.  Therefore, by $WF_0(t_k)$, $T_{\sigma_{t_k}}$ is well-founded.  Since $T_{\sigma_{t_k}}=T_{\sigma_t^\frown\langle k\rangle}$ is well-founded whenever $\sigma_t^\frown\langle k\rangle\in T$, it follows that $T_{\sigma_t}$ is well-founded, as desired.

Since $|\sigma_s|=i+1-r$ for some $r$, the statement holds in particular for $s$.
\end{proof}

\begin{lemma}[\ACA]
Let $\phi$ be a basic rank $n+1$ formula, let $s\in\mathcal{T}_{n+1}(T)$, and suppose that for every $t\supseteq s$ such that $t$ decides $\phi$, $WF_{n+1}(t)$.  Then $WF_{n+1}(s)$.
\label{decides_rank_successor}
\end{lemma}
\begin{proof}
Without loss of generality, we may assume $\phi$ is the formula $\exists x\psi$.  It suffices to show that whenever the assumption holds of $s$, $WF'_{n+1}(s)$.  If $s$ decides $\phi$ we have $WF_{n+1}(s)$ by assumption, so assume $s$ does not decide $\phi$.  Assume $s$ satisfies the premise of $WF'_{n+1}(s)$: whenever $t\supseteq\pi^{n+1}_n(s)$ and $t\prec^{+1}s$, $WF_n(t)$.

First, consider any $s_+\supseteq s$ such that $(s_+\setminus s)\cap\mathcal{F}_{n+1}=\{\phi\}$, so $s_+$ decides $\phi$.  Suppose $t\supseteq\pi^{n+1}_n(s_+)$ and $t\prec^{+1}s_+$.  Then there is a formula $\forall x\psi'\in s_+$ and a $k$ such that $\hat t\vdash\neg\psi'(k)$.  We must have $\forall x\psi'\in s$ and therefore $t\prec^{+1}s$, so $WF_n(t)$.  Since $WF_{n+1}(s_+)$ holds, it follows that whenever $t\supseteq \pi^{n+1}_n(s_+)$, $WF_n(t)$.

Now let $s_-=s\cup\{\forall x\neg\psi\}$ and suppose $t\supseteq\pi^{n+1}_n(s_-)$ and $t\prec^{+1}s_-$.  As before, there is a formula $\forall x\psi'\in s_-$ and a $k$ such that $t\vdash\psi'(k)$.  If $\psi'\neq\neg\psi$, again we have $WF_n(t)$ since $t\prec^{+1}s$.  Otherwise, set $s_+=t\cup s\cup\{\phi\}$; then $\pi^{n+1}_n(s_+)=t$, and therefore $WF_n(t)$ by the preceding paragraph.  So for any $t\supseteq\pi^{n+1}_n(s_-)$ with $t\prec^{+1}s_-$, $WF_n(t)$.  Since $s_-$ decides $\phi$, we have $WF_{n+1}(s_-)$, and therefore for all $t\supseteq\pi^{n+1}_n(s_-)$, $WF_n(t)$.  Since $\pi^{n+1}_n(s_-)=\pi^{n+1}_n(s)$, it follows that whenever $t\supseteq \pi^{n+1}_n(s)$, $WF_n(t)$, and therefore $WF'_{n+1}(s)$.
\end{proof}

We wish the previous lemma to hold even when $rk(\phi)<n$.  To do this we prove the following inductive step.
\begin{lemma}[\ACA]
Let $\phi$ be a basic rank $m$ formula, let $n\geq m$, and suppose that:
\begin{quote} 
Whenever $s\in\mathcal{T}_n(T)$ and for every $t\supseteq s$ such that $t$ decides $\phi$, $WF_n(t)$, then $WF_n(s)$.
\end{quote}
Then:
\begin{quote} 
Whenever $s\in\mathcal{T}_{n+1}(T)$ and for every $t\supseteq s$ such that $t$ decides $\phi$, $WF_{n+1}(t)$, then $WF_{n+1}(s)$.
\end{quote}
\label{decide_rank_increment}
\end{lemma}
\begin{proof} 
Let $s\in\mathcal{T}_{n+1}(T)$ be given, and suppose that for every $t\supseteq s$ such that $t$ decides $\phi$, $WF_{n+1}(t)$.  Again, it suffices to show that $WF'_{n+1}(s)$.  Suppose that whenever $t\supseteq\pi^{n+1}_n(s)$ and $t\prec^{+1}s$, $WF_n(t)$.  Let $t\supseteq\pi^{n+1}_n(s)$ be arbitrary; we will show $WF_n(t)$.  To do this, it suffices to show that whenever $t'\supseteq t$ decides $\phi$, $WF_n(t')$.

So suppose $t'\supseteq t$ is given such that $t'$ decides $\phi$.  If $t'\prec^{+1}s$ then $WF_{n}(t')$ by assumption.  Otherwise, set $s'=t'\cup s$.  Since $s'$ decides $\phi$, $WF_{n+1}(s')$ holds.  Whenever $t''\supseteq\pi^{n+1}_{n}(s')=t'$ with $t''\prec^{+1}s'$, also $t''\prec^{+1}s$, and therefore $WF_{n}(t'')$.  Therefore for any $t''\supseteq \pi^{n+1}_{n}(s')=t'$, $WF_{n}(t'')$, and in particular $WF_{n}(t')$.

\end{proof}

\begin{lemma}[\ACA]
If $WF_{n+1}(\emptyset)$ then $WF_n(\emptyset)$.
\label{wf_empty_decrement}
\end{lemma}
\begin{proof} 
If $WF_{n+1}(\emptyset)$ then, in particular, $WF'_{n+1}(\emptyset)$.  If $t\supseteq\pi^{n+1}_n(\emptyset)$ then we cannot have $t\prec^{+1}\emptyset$, so the premise of $WF'_{n+1}(\emptyset)$ is trivially satisfied, and therefore whenever $t\in \mathcal{T}_n(T)$, $t\supseteq\pi^{n+1}_n(\emptyset)=\emptyset$, so $WF_n(t)$.  In particular, $WF_n(\emptyset)$.
\end{proof}

\begin{definition} 
Given $s\in\mathcal{T}_n(T)$ and a formula $\phi(x,y)$ with only the displayed free variables, we define a sequence $\sigma_{s,\phi}$ recursively: $\emptyset\subseteq\sigma_{s,\phi}$, and if $\tau\sqsubseteq\sigma_{s,\phi}$ and there is exactly one $i$ such that $\hat s\vdash\phi(|\tau|,i)$ then $\tau^\frown\langle i\rangle\sqsubseteq\sigma_{s,\phi}$. 
\end{definition}

\begin{definition} 
Let $n$ be a successor\footnote{By a successor, we mean $n>0$.  In the remainder of this section, we will refer to numbers $>0$ as ``successors'' in definitions or theorems which will apply unchanged when we generalize to infinite well orderings.}.  We define $\widehat{\mathcal{T}_{n}(T)}$ to consist of pairs $(s,U)$ such that:
\begin{itemize} 
  \item $s\in\mathcal{T}_n(T)$
  \item $U$ is a partial functions whose domain is a finite set of basic formulas of rank $\leq n$ of the form $\exists z\phi(x,y,z)$ with only the displayed free variables such that $\sigma_{s,\phi}\prec \sigma_s$, and whose range is $\{0,1\}$
  \item If $U(\phi)$ is defined then $U(\phi)=1$ iff one of the following \emph{excluding conditions} holds:
     \begin{itemize} 
       \item There are $m,i,j$ with $i\neq j$ such that $\exists z\phi(m,i,z)\in s$ and $\exists z\phi(m,j,z)\in s$,
       \item There is an $m$ such that $\forall u\neg\phi(m,p_1(u),p_2(u))\in s$, or
       \item $\sigma_{s,\phi}\not\in T$.
     \end{itemize}
\end{itemize}

We say $(s,U)$ \emph{decides} $U(\phi)$ if $U(\phi)$ is defined.  We say $(s,U)\supseteq(t,V)$ if $s\supseteq t$, $\dom(U)\supseteq\dom(V)$, and $U\upharpoonright\dom(V)=V$.

We define $\hat\pi:\widehat{\mathcal{T}_{n}(T)}\rightarrow\mathcal{T}_n(T)$ by $\hat\pi(s,U)=s$.

If $t\in\mathcal{T}_n(T)$ and $(s,U)\in\widehat{\mathcal{T}_{n}(T)}$, we say $t\prec^{+1}(s,U)$ if there is a $\phi$ such that $U(\phi)=0$ but $t$ satisfies one of the above excluding conditions for $\phi$.

$\widehat{WF}'_{n}(s,U)$ holds if
    \begin{quote}
    Suppose that for all $t\supseteq s$ such that $t\prec^{+1}(s,U)$, $WF_n(t)$; then for all $t\supseteq\hat\pi(s,U)$, $WF_n(t)$.
    \end{quote}
$\widehat{WF}_{n}(t,V)$ holds if for all $(s,U)\supseteq (t,V)$, $\widehat{WF}'_{n}(s,U)$ holds.
\end{definition}
These definitions are very similar to the $n+1$ cases above; in place of existential formulas, we have ``witnesses that $\phi$ fails to define a path to the left of the official path''.

As above, we have
\begin{lemma}[\ACA]
If $n$ is a successor,
\begin{enumerate} 
  \item If $\widehat{WF}_{n}(t,V)$ holds and $(s,U)\supseteq (t,V)$ then $\widehat{WF}_{n}(s,U)$ holds.
  \item If $WF_n(s)$ holds then $\widehat{WF}_{n}(s,U)$ holds.
  \item If $\widehat{WF}_{n}(\emptyset,\emptyset)$ then $WF_n(\emptyset)$.
\end{enumerate}
\end{lemma}

\begin{lemma}[\ACA]
Let $n$ be a successor.  Let $\phi$ be a formula and suppose that whenever $(t,V)\supseteq(s,U)$ and $(t,V)$ decides $U(\phi)$, $\widehat{WF}_{n}(t,V)$.  Then $\widehat{WF}_{n}(s,U)$.
\end{lemma}
\begin{proof} 
The proof is similar to that of Lemma \ref{decides_rank_successor}.  It suffices to show that $\widehat{WF}'_{n}(s,U)$.  Suppose the premise of $\widehat{WF}'_{n}(s,U)$ holds, so that whenever $s'\supseteq s$ and $s'\prec^{+1}(s,U)$, $WF_n(s')$.

First, if $s$ satisfies one of the excluding conditions for $\phi$ then $(s,U\cup\{(\phi,1)\})\in\widehat{\mathcal{T}_n(T)}$, and therefore $\widehat{WF}_n(s,U\cup\{(\phi,1)\})$.  Since $\hat\pi(s,U\cup\{(\phi,1)\})=\hat\pi(s,U)$, we have $\widehat{WF}_n(s,U)$.

Otherwise, let $(t,V)$ extend $(s,U)$ such that $V(\phi)=1$ and $\dom(V)\setminus\dom(U)=\{\phi\}$, and let $t'\supseteq t$ with $t'\prec^{+1}(t,V)$.  Then there is a $\psi$ such that $t'$ satisfies one of the excluding conditions for $\psi$ but $V(\psi)=0$.  Therefore $U(\psi)=0$ as well, so $t'\prec^{+1}(s,U)$, and therefore $WF_n(t')$.  Since $\widehat{WF}_{n}(t,V)$ holds, it follows that for all $t'\supseteq\hat\pi(t,V)$, we have $WF_n(t')$.

Now set $V=U\cup\{(\phi,0)\}$, so $(s,V)$ decides $\phi$, and let $t\supseteq s$ with $t\prec^{+1}(s,V)$.  Then there is a $\psi$ such that $t$ satisfies one of the excluding conditions for $\psi$ but $V(\psi)=0$.  If $\psi\neq\phi$ then $t\prec^{+1}(s,U)$, and therefore $WF_n(t)$.  If $\psi=\phi$ then $(t,U\cup\{(\phi,1)\})$ also decides $\phi$, and so we have shown in previous paragraph that again $WF_n(t)$.  Since $\widehat{WF}_{n}(s,V)$, it follows that whenever $t\supseteq s$, $WF_n(s)$, as desired.
\end{proof}

\begin{lemma}[\ACA] 
Let $\phi$ be a formula and suppose that whenever $s\in\mathcal{T}_{n}$ and for every $t\supseteq s$ such that $t$ decides $\phi$, $WF_{n}(t)$, then $WF_{n}(s)$.  Then whenever $(s,U)\in\widehat{\mathcal{T}_n(T)}$ is such that for every $(t,V)\supseteq(s,U)$ such that $(t,V)$ decides $\phi$, $\widehat{WF}_n(t,V)$, then $\widehat{WF}_n(s,U)$.
\end{lemma}
\begin{proof} 
The proof is similar to that of Lemma \ref{decide_rank_increment}.  It will suffice to show $\widehat{WF}'_{n}(s,U)$.  Suppose the premise of $\widehat{WF}'_{n}(s,U)$ holds, so whenever $t\supseteq s$ and $t\prec^{+1}s$, $WF_n(t)$.

Let $t\supseteq s$ be given.  By assumption, it suffices to show that whenever $t'\supseteq t$ and $t'$ decides $\phi$, $WF_n(t')$.  So let some $t'\supseteq t$ be given such that $t'$ decides $\phi$.  If $t'\prec^{+1}(s,U)$ then $WF_n(t')$.  Otherwise $(t',U)\supseteq (s,U)$ and decides $\phi$, so $\widehat{WF}_{n}(t',U)$.  Moreover, whenever $t''\supseteq t'$ and $t''\prec^{+1}(t',U)$, $t''\prec^{+1}(s,U)$, and therefore $WF_n(t'')$.  So we have $WF_n(t')$, as desired.
\end{proof}

By a \emph{decision of rank $n$}, we mean either a formula $\phi$ of rank $n$, or $U(\phi)$ where $\phi$ is a basic formula of rank $\leq n$ in the form $\exists z\phi(x,y,z)$ with only the displayed free variables.  Note that if $s$ decides $\phi$ or $U(\phi)$ and $t\supseteq s$ then $t$ decides $\phi$ or $U(\phi)$ as well; therefore we may say an infinite path decides $\phi$ or $U(\phi)$ if any finite initial segment does.

The following lemma is a modification of the usual statement that when $\mathcal{S}$ is a well-founded subset of $\widehat{\mathcal{T}_{n}(T)}$ we can carry out transfinite induction along $\mathcal{S}$.
\begin{lemma}[\SPiti{m}]
Let $m$ be a successor.  Let $\mathcal{S}\subseteq\widehat{\mathcal{T}_{n}(T)}$, and suppose that there is no infinite path through $\mathcal{S}$ deciding every decision of rank $n$.  Let $A$ be a formula in $\Pi_m(\Pi^1_1)$ and suppose the following principle holds:
\begin{quote} 
For any $(s,U)\in\mathcal{S}$, if there is a decision $d$ such that whenever $(t,V)\supseteq(s,U)$, $(t,V)\in\mathcal{S}$, and $(t,V)$ decides $d$, $A(t,V)$ holds, then $A(s,U)$ holds.
\end{quote}
Then for every $(s,U)\in\mathcal{S}$, $A(s,U)$ holds.
\label{sigma_2_induction}
\end{lemma}
Despite the complicated statement, this lemma actually just reformulates transfinite induction in a convenient form.  Transfinite induction is usually stated for trees which are genuinely well-founded; equivalently, there are no paths satisfying a $\Pi_1$ property (namely, having elements at every level).  Here we restrict ourselves to those paths satisfying a $\Pi_2$ property---deciding every decision.  The fact that these are equivalent is essentially a consequence of the well-known fact that statements of the form $\forall X\exists y\forall z\phi$ (for $\phi$ quantifier-free) are equivalent to statements of the form $\forall X\exists y\phi'$.
\begin{proof} 
Fix a surjective function $\rho$ from $\mathbb{N}$ to the set of decisions, and consider the tree $\mathcal{S}'$ of increasing sequences $\sigma$ from $\mathcal{S}$ such that for each $i$, if $\rho(i)$ is a decision of rank $\leq n$ then $\sigma(i)$ decides $\rho(i)$.  Clearly any infinite path through $\mathcal{S}'$ gives an infinite path through $\mathcal{S}$ deciding all formulas, so $\mathcal{S}'$ is well-founded.

Let $A'(\sigma)$ hold if $A(\sigma(|\sigma|-1))$ holds, so $A'$ is a $\Pi_m(\Pi^1_1)$ formula.  We claim $A'$ is progressive: let $\sigma$ be given with $(s,U)$ its final element, and suppose that for all $(t,V)$ such that $\sigma^\frown\langle (t,V)\rangle\in\mathcal{S}'$, $A'(\sigma^\frown\langle (t,V)\rangle)$.  Then whenever $(t,V)\supseteq(s,U)$ and $(t,V)$ decides $\rho(|\sigma|)$, $A'(\sigma^\frown\langle (t,V)\rangle)$, and therefore $A(t,V)$.  Therefore $A(s,U)$, and so $A'(\sigma)$.

So by transfinite induction on $\mathcal{S}'$, $A'$ holds of all $\sigma$, and in particular, $A(s,U)$ for all $(s,U)\in\mathcal{S}$.
\end{proof}

\begin{definition}
\[\widehat{\mathcal{T}_{n}(T)}^+=\{(s,U)\in\widehat{\mathcal{T}_{n}(T)}\mid U(\phi)=1\text{ whenever }U(\phi)\text{ is defined}\}.\]
\end{definition}

The following lemma is stated with premises we have already shown to be true so that it can easily be adapted to the case where $n$ is replaced by an infinite well-ordering later.
\begin{theorem} [\SPiti{n+2}]
Let $n$ be a successor.  Suppose that:
\begin{itemize} 
  \item Whenever $WF_0(\pi^n_0(s))$, $\widehat{WF}_n(s,U)$,
  \item If $\widehat{WF}_n(\emptyset,\emptyset)$ then $WF_0(\emptyset)$,
  \item There is no infinite path through $\widehat{\mathcal{T}_{n}(T)}^+$ deciding every decision. 
\end{itemize}
Then $T$ is well-founded.
\label{arithmetic_well_founded}
\end{theorem}
\begin{proof} 
We first show that $\widehat{WF}_{n}(s,U)$ holds for all $(s,U)\in\widehat{\mathcal{T}_{n}(T)}\setminus\widehat{\mathcal{T}_{n}(T)}^+$.  Let $(s,U)$ be given with $U(\phi)=0$ for some $\phi$.  If $T_{\sigma_{s,\phi}}$ is ill-founded then $\sigma_{s,\phi}\prec\sigma_s$ witnesses $WF_0(\pi^n_0(s))$, and therefore $\widehat{WF}_{n}(s,U)$.  Otherwise, for each $\tau\in T_{\sigma_{s,\phi}}$, let
\[\mathcal{S}_\tau=\{(t,V)\supseteq(s,U)\mid \sigma_{t,\phi}=\tau\}.\]
We proceed by induction on $\tau\in T_{\sigma_{s,\phi}}$ showing that for every $(t,V)\in\mathcal{S}_\tau$, $\widehat{WF}_{n}(t,V)$.

Let $(t,V)\in\mathcal{S}_\tau$ be given and suppose that for every $k$ such that $\sigma_{t,\phi}{}^\frown\langle k\rangle\in T$ and every $(t',V')\supseteq(s,U)$ with $\sigma_{t',\phi}=\sigma_{t,\phi}{}^\frown\langle k\rangle$, $\widehat{WF}_{n}(t',V')$.  Let $(t',V')$ be any extension of $(t,V)$ deciding $\exists k\phi(|\sigma_t|,k)$ (note that, pairing variables and using the fact that $n$ is a successor, this has the same rank as $\phi$).  Then since $V'(\phi)=V(\phi)=U(\phi)=0$, it must be that there is such a $k$, and therefore $t'\vdash\phi(|\sigma_t|,k)$ for some $k$ and $\sigma_{t',\phi}\in T$, so $\widehat{WF}_{n}(t',V')$.  So $\widehat{WF}_{n}(t',V')$ holds for any $(t',V')\supseteq(t,V)$ deciding $\exists k\phi(|\sigma_t|,k)$, and therefore $\widehat{WF}_{n}(t,V)$ holds.

By induction, for any $\tau\in T_{\sigma_{s,\phi}}$, any $(t,V)\in\mathcal{S}_\tau$ satisfies $\widehat{WF}_{n}(t,V)$.  In particular, $\widehat{WF}_{n}(s,U)$.

Now we show that $\widehat{WF}_{n}(s,U)$ holds for $\widehat{\mathcal{T}_{n}(T)}^+$ using the modified induction given by the previous lemma.  Suppose that $(s,U)\in\widehat{\mathcal{T}_{n}(T)}^+$ and $\psi$ is a decision such that whenever $(t,V)\supseteq(s,U)$, $(t,V)\in\widehat{\mathcal{T}_{n}(T)}^+$, and $(t,V)$ decides $\psi$, $\widehat{WF}_{n}(t,V)$.  Then for any $(t,V)\supseteq(S,U)$ deciding $\psi$, either $(t,V)\in\widehat{\mathcal{T}_{n}(T)}^+$, in which case $\widehat{WF}_{n}(t,V)$ by assumption, or $(t,V)\not\in\widehat{\mathcal{T}_{n}(T)}^+$, in which case we have just shown that $\widehat{WF}_{n}(t,V)$.  Therefore, by the previous lemma, for every $(s,U)\in\widehat{\mathcal{T}_{n}(T)}^+$, $\widehat{WF}_{n}(s,U)$.

It follows in particular that $\widehat{WF}_{n}(\emptyset,\emptyset)$, and therefore $WF_0(\emptyset)$.  Since there are no $\sigma\in T$ with $\sigma\prec\emptyset$, it follows that $T_\emptyset=T$ is well-founded.
\end{proof}

\begin{definition} 
Let $\phi$ be a closed formula of $\mathcal{L}$.  Then $\hat\phi(X,Y)$ is the formula of second-order arithmetic which interprets the function symbol $F$ by $X$ and the predicate symbol $\hat{T}$ by $Y$.
\end{definition}

\begin{lemma}[\ACA]
Let $\Lambda$ be a path through $\mathcal{T}_n(T)$ deciding all formulas of rank $\leq n$ and let $\sigma_\Lambda$ be the corresponding sequence through $T$ given by $\sigma_\Lambda(i)=j$ iff $F(i)=j\in\Lambda(m)$ for some (and therefore cofinitely many) $m$.

Then whenever $\phi$ is a closed formula of rank $\leq n$, the following are equivalent:
\begin{enumerate} 
  \item There is an $m$ such that $\Lambda(m)\vdash\phi$,
  \item $\hat\phi(\sigma_\Lambda,T)$.
\end{enumerate}
\end{lemma}
\begin{proof} 
We proceed by induction on formulas.  When $\phi$ is atomic, the equivalence follows immediately from the definitions.

Suppose the claim holds for $\phi$ and $\psi$.  The claim for $\neg\phi$ follows from the equivalence for $\phi$ and the fact that $\Lambda$ decides all formulas of rank $\leq n$, including $\neg\phi$.  Similarly for other propositional combinations of $\phi$ and $\psi$.

Suppose that for every $k$, the claim holds for $\phi(k)$.  If $\exists x\phi\in\Lambda(m)$ then there is some $k$ such that $\Lambda(m)\vdash\phi(k)$, and by IH, $\hat\phi(k)(\sigma_\Lambda,T)$, and therefore $\widehat{\exists x\phi}(\sigma_\Lambda,T)$.  If $\forall x\phi\in\Lambda(m)$ then there are no $m',k$ such that $\Lambda(m')\vdash\neg\phi(k)$, and therefore $\widehat{\neg\phi(k)}(\sigma_\Lambda,T)$ never holds, so $\widehat{\forall x\phi}(\sigma_\Lambda,T)$ holds.  The other direction follows since either $\widehat{\exists x\phi}(\sigma_\Lambda,T)$ or $\widehat{\forall x\neg\phi}(\sigma_\Lambda,T)$ must hold, and there is some $m$ such that either $\exists x\phi\in\Lambda(m)$ or $\forall x\neg\phi\in\Lambda(m)$. 
\end{proof}

\begin{theorem}
For any finite $n>0$, \SPiti{n+2} implies \RLPP{n}.
\label{spiti_rmpp}
\end{theorem}
\begin{proof} 
Let $T$ be a tree of finite sequences and let $\prec$ be a well-founded partial order.  Suppose that for every path $\Lambda$ through $T$, there is a $\Lambda'$ which is $\Sigma_n$ in $T\oplus\Lambda$ with $\Lambda'\prec\Lambda$.

Suppose there were an infinite path $\Lambda$ through $\widehat{\mathcal{T}_n(T)}^+$ deciding every decision of rank $\leq n$.  For each $i$, there is a unique $\sigma_\Lambda(i)$ such that $F(i)=\sigma_\Lambda(i)\in \hat\pi(\Lambda(j))$ for some $j$ (and therefore cofinitely many $j$).  The function $\sigma_\Lambda$ must be a path through $T$.  Suppose $\Lambda'\prec\sigma_\Lambda$ and there is a $\Sigma_n$ formula $\phi$ such that $\exists z \phi(i,j,z,T,\sigma_\Lambda)$ iff $\Lambda'(i)=j$.  By the previous lemma, we have $\exists z\phi(i,j,z)\in\Lambda(m)$ for some $m$ iff $\Lambda'(i)=j$.  Since $\Lambda'$ is a path through $T$, none of the excluding conditions for $\phi$ can ever hold, so whenever $\Lambda(j)=(s,U)$ and $U(\phi)$ is defined, $U(\phi)=0$.  But this would contradict the fact that $\Lambda$ is a path through $\widehat{\mathcal{T}_n(T)}^+$.  So there is no infinite path $\Lambda$ through $\widehat{\mathcal{T}_n(T)}^+$ deciding every decision.

Observe that the first two conditions in Theorem \ref{arithmetic_well_founded} all hold in \SPiti{n+2} (since $n$ is finite), so we obtain the conclusion that $T$ is well-founded.
\end{proof}

\section{The Nash-Williams Theorem and Menger's Theorem}
\subsection{The Nash-Williams Theorem}
In what follows, we will use the letter $b$ (and variants $b'$ and so on) to represent finite sequences which are intended to be increasing (and specifically, members of a barrier).  We briefly define the key notions needed to state and prove the Nash-Williams Theorem; a more careful exposition is found in \cite{marcone:MR1428011}.
\begin{definition} 
A sequence $b$ is \emph{increasing} if whenever $i<j$, $b(i)<b(j)$.

Let $B$ be a set of finite increasing sequences.  We write $base(B)$ for the set of $n$ such that for some $b\in B$ and some $i\in\dom(b)$, $b(i)=n$.

A \emph{barrier} is a set $B$ of finite increasing sequences such that:
\begin{itemize} 
  \item $base(B)$ is infinite,
  \item If $\Lambda$ is an infinite increasing sequence from $base(B)$, there is a $b\in B$ such that $b\sqsubset\Lambda$,
  \item If $b,b'\in B$ and $b\neq b'$ then $\rng(b)\not\subseteq \rng(b')$
\end{itemize}

If $b$ is a non-empty sequence, we write $b^-$ for the sequence with $|b^-|=|b|-1$ given by $b^-(i)=b(i+1)$ (and $b^-(i)$ is undefined if $b(i+1)$ is).

If $b,b'$ are sequences, we write $b\triangleleft b'$ if there is a $ b^*$ such that $b\sqsubseteq b^*$ and $b'\sqsubseteq (b^*)^-$.

Let $Q$ be a partial order.  If $B$ is a barrier, $B'\subseteq B$ (where $B'$ is finite or infinite), and $f:B'\rightarrow Q$, $f$ is \emph{good} if for some $b,b'\in B'$ with $b\triangleleft b'$, $f(b)\preceq_Q f(b')$.  If $f$ is not good, $f$ is \emph{bad}.  If for every $b,b'\in B'$ with $b\triangleleft b'$, $f(b)\preceq_Q f(b')$ then $f$ is \emph{perfect}.

If $B$ is a barrier, $Q$ is a \emph{$B$-better-quasi-order} ($B$-bqo) if for every barrier $B'\subseteq B$ and every $f:B'\rightarrow Q$, $f$ is good.  $Q$ is a \emph{better-quasi-order} (bqo) if for every barrier $B$, $Q$ is a $B$-bqo.
\end{definition}
$B$ is a barrier iff $\{b\mid \forall b'\in B\ b'\not\sqsubseteq b\}$ is well-founded as a tree from $base(B)$.

\begin{definition} 
Given $Q$, $\tilde Q$ is the class of all pairs $(\alpha,f)$ where $\alpha$ is a well-order and $f:\alpha\rightarrow Q$.  If $(\alpha,f),(\beta,g)\in\tilde Q$, we say $(\alpha,f)\tilde\preceq_Q(\beta,g)$ if there is a strictly increasing function $\pi:\alpha\rightarrow\beta$ such that for all $\gamma\in\alpha$, $f(\gamma)\preceq_Q g(\pi(\gamma))$.

{\NWT}, the Nash-Williams Theorem, is the statement that for if $Q$ is a bqo then $\tilde Q$ is bqo.\footnote{Note that even though $\tilde Q$ is not a set, we can still formulate the statement that $\tilde Q$ is bqo in second order arithmetic.}

{\GHT}, the Generalized Higman's Theorem, is the statement that if $Q$ is a $B$-bqo then $Q^{<\omega}$ is a $B$-bqo.\footnote{Our statement of {\GHT} differs slightly from Marcone's: Marcone takes {\GHT} to be the statement that if $Q$ is a $B$-bqo for all barriers $B$ then $Q^{<\omega}$ is a $B$-bqo for all barriers $B$, which is a $\Pi^1_3$ statement, and mentions this version of {\GHT} as an intermediate step.}
\end{definition}

Marcone \cite{marcone:MR1428011} has shown:
\begin{theorem}
\begin{enumerate}
  \item In {\ATR}, {\NWT} is equivalent to {\GHT}.
  \item {\Pioo} implies {\GHT}.
\end{enumerate}
\end{theorem}

Since {\GHT} is a $\Pi^1_2$ sentence, it is not possible for {\GHT} be equivalent to {\Pioo}.  We will now show that Marcone's proof goes through essentially unchanged in {\TMPP}.\footnote{Marcone's proof uses the ``locally minimal bad array lemma'', which is a principle similar, and equivalent, to the minimal path principle.  This lemma is essentially an encapsulation of the particular application of the minimal path principle we use below.  Another family of relative principles---the relatively locally minimal bad array lemma and so on---could be defined, but since they would be minor combinatorial variants on the principles we have given, we do not do so.}

\begin{definition} 
If $B$ is a barrier and $X\subseteq\mathbb{N}$, we write $B\upharpoonright X$ for $\{b\in B\mid \rng(b)\subseteq X\}$.
\end{definition}

\begin{lemma} [\RCA]
If $X$ is an infinite subset of $base(B)$ then $B\upharpoonright X$ is a barrier.
\end{lemma}
\begin{proof} 
Clearly $base(B\upharpoonright X)\subseteq X$.  Suppose $\Lambda$ is an infinite increasing sequence from $X$.  Then since $X\subseteq base(B)$, there is a $b\in B$ such that $b\sqsubset\Lambda$, and therefore $b\in B\upharpoonright X$.  Therefore $base(B\upharpoonright X)=X$ and every infinite sequence through $X$ has an initial segment in $B\upharpoonright X$.  The other two conditions are immediate since $B\upharpoonright X\subseteq B$.
\end{proof}

\begin{lemma} [\RCA]
If $B'\subseteq B$ is a barrier then $B'=B\upharpoonright base(B')$.
\end{lemma}
\begin{proof} 
Suppose $b\in B\upharpoonright base(B')$.  Let $\Lambda$ be an infinite increasing sequence from $base(B')$ such that $b\sqsubset\Lambda$.  Then there is a $b'\in B'\subseteq B$ such that $b'\sqsubset\Lambda$.  If $b'\neq b$ then we have either $\rng(b)\subseteq\rng(b')$ or $\rng(b')\subseteq\rng(b)$, contradicting the fact that $B$ is a barrier.
\end{proof}

\begin{definition} 
A sequence $\sigma$ from $B\times Q$, $\sigma=\langle(b_0,q_0),\ldots,(b_k,q_k)\rangle$ is a \emph{bad partial array} if:
\begin{itemize} 
  \item When $i<j$, $\max b_i\leq \max b_j$,
  \item If $b_i\triangleleft b_j$, $q_i\not\preceq_Q q_j$,
  \item If $b\in B\upharpoonright base(\{b_0,\ldots,b_k\})$ and $\max b<\max b_k$ then there is an $i<k$ such that $b=b_i$.
\end{itemize}

If $\sigma$ is a bad partial array, we define a partial function $f^\sigma:B\rightarrow Q$ by setting $f^\sigma(b)=q$ iff there is an $i$ such that $\sigma(i)=(b,q)$.  If $\Lambda$ is an infinite path through the tree of bad partial arrays, we define $f^\Lambda$ similarly.
\end{definition}

\begin{lemma} [\RCA]
$f$ is a bad function from a barrier $B'\subseteq B$ to $Q$ iff there is an infinite path $\Lambda$ through the tree of bad partial arrays such that $f=f^\Lambda$.
\end{lemma}
\begin{proof} 
Suppose $B'\subseteq B$ is a barrier and $f:B'\rightarrow Q$ is bad.  Fix an enumeration of $B'$, $B'=\{b_0,b_1,\ldots\}$ such that if $i<j$ then $\max b_i\leq\max b_j$.  Define $\Lambda(i)=(b_i,f(b_i))$.  Clearly $f=f^\Lambda$.  We must check that if $\sigma\sqsubset\Lambda$ then $\sigma$ is a bad partial array; the first two conditions are immediate from the enumeration of $B'$ and the fact that $f$ is bad.  If $b\in B\upharpoonright base(\{b_0,\ldots,b_k\})$ and $\max b<\max b_k$ then $b\in B\upharpoonright base(B')=B'$, so there is an $i<k$ such that $b=b_i$.

Suppose $\Lambda$ is an infinite path through the tree of bad partial arrays.  Then $f^\Lambda$ is clearly bad, and we must check that $\dom(f^\Lambda)$ is a barrier.  If $b\in B\upharpoonright base(\dom(f^\Lambda))$ then there must be some $n$ such that $b\in base(\dom(f^{\Lambda\upharpoonright n}))$ and $\max b<base(\dom(f^{\Lambda\upharpoonright n}))$, which implies that $b\in\dom(f^{\Lambda\upharpoonright n})$.
\end{proof}

We use the following uniformly effective version of the clopen Ramsey Theorem (the proof of the clopen Ramsey Theorem in {\ATR} is carried out in \cite{friedman:MR694261}, with another proof given in \cite{avigad98}; the effective bounds for the complexity are given in \cite{clote:MR745367}).
\begin{theorem} [\ATR]
For each barrier $B$, there is an ordinal $\alpha$ such that whenever $B'\subseteq B$ is a barrier and $c:B'\rightarrow \{0,1\}$, there is an infinite $S\subseteq base(B)$ such that $S$ is $\Sigma_\alpha$ in $c\oplus B'$ and $c$ is constant on $B\upharpoonright S$.
\end{theorem}

We may adapt colorings of $B$ to colorings of pairs from $B$:
\begin{lemma} [\ATR]
For each barrier $B$, there is an ordinal $\alpha$ such that whenever $B'\subseteq B$ is a barrier and $c:\{b,b'\in B'\mid b\triangleleft b'\}\rightarrow\{0,1\}$, there is an infinite $S\subseteq base(B)$ such that $S$ is $\Sigma_\alpha$ in $c\oplus B'$ and $c$ is constant on $\{b,b'\in B\upharpoonright S\mid b\triangleleft b'\}$.
\end{lemma}
\begin{proof} 
Whenever $b,b'\in B$ with $b\triangleleft b'$, write $b\cup b'$ for the unique increasing sequence such $\rng(b\cup b')=\rng(b)\cup\rng(b')$.  Define $B^*=\{b\cup b'\mid b,b'\in B\text{ and }b\triangleleft b'\}$.  Note that $B^*$ is a barrier on $base(B)$: if $\Lambda$ is an infinite increasing sequence from $base(B)$, we may find some $b\in B$ such that $b\sqsubset\Lambda$.  We may also find some $b'\in B$ such that $b'\sqsubseteq\Lambda^-$.  Then we have $b\triangleleft b'$, and therefore $b\cup b'\sqsubseteq\Lambda^-$ and $b\cup b'\in B^*$.  Let $\alpha$ be such that given any coloring of a subbarrier of $B^*$, there is an infinite homogeneous subbarrier $\Sigma_\alpha$ in $c\oplus B^*$.

Now let $B'\subseteq B$ be given and let $c$ be a coloring of $\{b,b'\in B'\mid b\triangleleft b'\}$.  We may define a coloring $c^*$ on $(B')^*\subseteq B^*$ by $c^*(b\cup b')=c(b,b')$.  Let $S\subseteq base(B')$ be given such that $c^*$ restricted to $(B')^*\upharpoonright S$ is constant.  Then $c$ restricted to $B'\upharpoonright S$ is constant.
\end{proof}

\begin{theorem} [\TLPP]
{\GHT} holds.
\label{tmpp_ght}
\end{theorem}
\begin{proof} 
Let $B$ be a barrier, and suppose $Q$ is a $B$-bqo.  We set $(b,\sigma)\prec (b',\sigma')$ if $|\sigma|<|\sigma'|$; clearly $\prec$ is a well-order on $B\times Q^{<\omega}$.  Suppose $Q^{<\omega}$ is not a $B$-bqo; then let $\Lambda$ be a relatively minimal infinite sequence through the tree of bad partial arrays from $B$ to $Q^{<\omega}$.

  Clearly $f^\Lambda(b)\neq\langle\rangle$ for all $b$, so we may write $f^\Lambda(b)=g(b)^\frown\langle q(b)\rangle$ for all $b\in\dom(f^\Lambda)$.  For $b\triangleleft b'$, define $c(b,b')=0$ iff $q(b)\preceq_Q q(b')$ and $c(b,b')=1$ otherwise.  By the previous lemma, there is an infinite $S\subseteq \dom(f^\Lambda)\subseteq B$ such that $c$ is constant on $B\upharpoonright S$---that is, $q$ restricted to $S$ is either bad or perfect.  Since $Q$ is a $B$-bqo, $c$ must be constantly $0$, so $q\upharpoonright (B\upharpoonright S)$ is perfect.

For each $n$, write $\Lambda(n)=(b_n,\sigma_n)$.  Let $n$ be least such that $b_n\in B\upharpoonright S$, and define $B^*=B\upharpoonright (S\cup base(\{b_i\}_{i<n}))$.  If $b_i\in B\upharpoonright S$, define $\Lambda'_0(i)=(b_i,g(b_i))$, and if $b_i\in B^*\setminus (B\upharpoonright S)$, define $\Lambda'_0(i)=(b_i,f(b_i))$; if neither of these apply, $\Lambda'_0(i)$ is undefined.  Let $\Lambda'$ be the infinite sequence defined recursively by $\Lambda'(i)=\Lambda'_0(j)$ where $j$ is least such that $\Lambda'_0(j)$ is defined and there is no $i'<i$ with $\Lambda'(i')=\Lambda'_0(j)$.

Observe that for $i<n$, $\Lambda'(i)=\Lambda(i)$ (since by construction, for $i<n$, $b_i\in B^*\setminus (B\upharpoonright S)$), and that $\Lambda'(n)\prec\Lambda(n)$ (since $\Lambda'(n)=(b_n,g(b_n))$ while $\Lambda(n)=(b_n,g(b_n)^\frown\langle q(b_n)\rangle)$).

We now show that $\Lambda'$ is an infinite path through the tree of bad sequences.  Since $\dom(f^{\Lambda'})=B^*$ is a barrier, we need only show that $f^{\Lambda'}$ is bad.  Suppose $\Lambda'(i)=(b,\sigma)$, $\Lambda'(j)=(b',\sigma')$, and $b\triangleleft b'$.

We consider three cases.  If $b\in B\upharpoonright S$ and $b'\in B\upharpoonright S$ then $\sigma=g(b)$, $\sigma'=g(b')$.  Since $q\upharpoonright (B\upharpoonright S)$ is perfect, $q(b)\preceq_Q q(B')$, and since $g(b)^\frown\langle q(b)\rangle\not\preceq^{<\omega}_Q g(b')^\frown\langle q(b')\rangle$, we must have $g(b)\not\preceq^{<\omega}_Q g(b')$.

If $b\in B^*\setminus (B\upharpoonright S)$ and $b'\in B^*\setminus (B\upharpoonright S)$ then $\sigma=f(b)$, $\sigma'=f(b')$, and we have $f(b)\not\preceq^{<\omega}_Q f(b')$.

Observe that for any $b\in B\upharpoonright S$, $\max (base(B^*)\setminus S)<\max b$.  In particular, this means that it is not possible to have $b\in B\upharpoonright S$ but $b'\in B^*\setminus (B\upharpoonright S)$.  The remaining case is that $b\in B^*\setminus (B\upharpoonright S)$ while $b'\in B\upharpoonright S$.  In this case we have $\sigma=f(b)$ while $\sigma'=g(b')$.  Since $g(b')\preceq^{<\omega}_Qf(b')$ and $f(b)\not\preceq^{<\omega}_Qf(b')$, we have $f(b)\not\preceq^{<\omega}_Qg(b')$.

Therefore $\Lambda'$ is an infinite sequence through the tree of bad partial arrays and $\Lambda'\prec\Lambda$.

It remains to check that the proof just given goes through in {\TLPP}.  It suffices to show that for each $B$, there is an $\alpha$ such that $\Lambda'$ is $\Sigma_\alpha$ in $\Lambda$.  Since $\Lambda'$ is computable from the set $S$, this follows from the fact that the coloring $c$ is computable from $\Lambda$ and there is an $\alpha$ such that $S$ is always $\Sigma_\alpha$ in $c$.
\end{proof}

\begin{corollary} 
{\NWT} holds in {\TLPP}.
\label{tmpp_nwt}
\end{corollary}

\subsection{Menger's Theorem}
In this subsection, we discuss a theorem about graphs.  When $G$ is a graph, we write $V(G)$ for the set of vertices and $E(G)$ for the set of edges.

\begin{definition} 
If $G$ is a graph and $A\subseteq V(G), B\subseteq V(G)$, an \emph{$A\mhyphen B$ path} is a finite sequences of vertices $v_0,\ldots,v_n$ such that $v_0\in A$, $v_n\in B$, and for each $i<n$, $(v_i,v_{i+1})\in E(G)$.

An \emph{$A\mhyphen B$ separator} is a set $C\subseteq V(G)$ such that every $A\mhyphen B$ path in $G$ contains an element of $C$.
\end{definition}

Menger's Theorem for countable graphs is:
\begin{theorem} 
For any $G$ and any $A,B\subseteq V(G)$, there is a set $M$ of disjoint $A\mhyphen B$ paths and an $A\mhyphen B$ separator $C$ such that $C$ consists of exactly one vertex from each path in $M$.
\end{theorem}

The proof uses the following notions:
\begin{definition} 
A \emph{warp} in $(G,A,B)$ is a subgraph $W$ of $G$ such that:
\begin{itemize} 
  \item $A\subseteq V(W)$,
  \item $W$ is a union of disjoint paths beginning in $A$.
\end{itemize}

If $W$ is a warp in $(G,A,B)$, $ter(W)$ is the set of vertices in $V(W)$ which are the terminal elements of paths beginning in $A$.

A warp $W$ is a \emph{wave} if $ter(W)$ is an $A\mhyphen B$ separator.

We order waves by $W\leq Y$ if $W$ is a subgraph of $Y$.
\end{definition}
It will be convenient to assume that our warps and waves do not contain elements of $A$ except as the first element of a path.

Shafer \cite{MR2899698} shows that Menger's Theorem for countable graphs is provable in {\Pioo}, and our treatment of Menger's Theorem follows his paper.  We will show
\begin{theorem} [\TLPP]
\label{tmpp_menger}
Menger's Theorem for countable graphs holds.
\end{theorem}

His proof is split into two parts:
\begin{lemma} [\Pioo]
For any graph $G$ and sets $A,B\subseteq V(G)$, there is a countably coded $\omega$-model $M$ of $\SDC$ containing $G,A,B$ such that $M$ believes there is a maximal (with respect to $\leq$) wave $W$.
\end{lemma}

\begin{lemma} [\ACA]
If $M$ is a countably coded $\omega$-model of $\SDC$ containing $G,A,B$ and $M$ believes there is a maximal wave $W$ then the conclusion of Menger's Theorem holds for $G,A,B$.
\end{lemma}

It suffices to show that the first lemma can be proven in {\TLPP}; by \ref{models_of_dc}, we need only show the following:
\begin{lemma} [\ACA]
Let $G,A,B$ be given.  There is an ill-founded tree $T$, a well-ordering $\prec$, and a computable bijection $\pi$ between waves and paths through $T$ such that whenever $W\leq W'$, $\pi(W')\preceq\pi(W)$.
\end{lemma}
\begin{proof} 
Fix an enumeration $V(G)=\{g_0,g_1,\ldots\}$ and an enumeration $\{p_0,p_1,\ldots\}$ of all $A\mhyphen{}B$ paths in $G$.  We define $T$ to consist of sequences $\langle \delta_0,\ldots,\delta_n\rangle$ such that:
\begin{itemize}
  \item If $k=2i$ then either $\delta_{k}=(0,0)$ or $\delta_{k}=(1,q_i)$ where $q_i$ is a path beginning with $A$ and ending with $g_i$,
  \item If $k=2k+1$ then $\delta_{k}=(i+2,S_i)$ where $S_i$ is a non-empty subset of $V(p_i)$, 
  \item If $q_i$ intersects $q_j$ then either $q_i$ is an end-extension of $q_j$ or $q_j$ is an end-extension of $q_i$,
  \item If $g_i$ appears in $q_j$ then $\delta_{2i}\neq (0,0)$,
  \item If $g_i\in A$ then $\delta_{2i}\neq (0,0)$,
  \item If $g_i\in S_j$ then $\delta_{2i}\neq (0,0)$,
  \item There is some $g_i\in S_j$ such that no path $q_k$ is a proper end-extension of $q_i$.
\end{itemize}

Given an infinite path $\Lambda$, we define a warp $W=\pi^{-1}(\Lambda)=\bigcup_i q_i$.  If $g_i\in A$ then $q_i$ witnesses that $g_i\in V(W)$.  $W$ is, by definition, a union of paths beginning in $A$, and the third condition ensures that distinct paths are disjoint.  To see that $W$ is a wave, observe that for any $A\mhyphen B$ path $p_i$, some element in $S_i$ must be the final element of a path.

Conversely, given a wave $W$, we define a path $\pi(W)$ through this tree as follows:
\begin{itemize} 
  \item If $g_i\in V(W)$ then $\pi(W)(2i)=(1,q_i)$ where $q_i$ is the (unique) path in $W$ beginning in $A$ and ending with $g_i$,
  \item If $g_i\not\in V(W)$ then $\pi(W)(2i)=(0,0)$,
  \item $\pi(W)(2i+1)=(i+2,V(p_i)\cap V(W))$.
\end{itemize}
Since $W$ is a wave, $\pi(W)$ is a path through $T$.

We define $\prec$ by:
\begin{itemize} 
  \item $(1,q)\prec(0,0)$,
  \item $(i+2,S)\prec(i+2,S')$ if $S'\subsetneq S$.
\end{itemize}
To see that this is well-founded, note in $(i+2,S)$, $|S|\leq |V(p_i)|$ is finite.

We must check that if $W< W'$ then $\pi(W')\prec\pi(W)$.  Since $W<W'$, there must be some $g_i\in V(W)\setminus V(W')$; we may assume $g_i$ is the least such.  Clearly $\pi(W')(2i)\prec\pi(W)(2i)$, so we need only check that for $j<i$, $\pi(W')(j)\preceq\pi(W)(j)$.  For $j$ even, by construction and the fact that $i$ was chosen least, $\pi(W)(j)=\pi(W')(j)$.  For $j$ odd, since $V(W')\subseteq V(W)$, we must have $\pi(W)(j)\preceq\pi(W')(j)$ as desired.

To see that $T$ is ill-founded, observe that there is a wave $W$ (specifically, $V(W)=A$ and $E(W)=\emptyset$), and therefore $\pi(W)$ is an infinite path through $T$.
\end{proof}

\section{The Relative Leftmost Path Principle}\label{ordinal_rpp}
In this section we prove:
\begin{theorem}[\ATR] \label{spiti_to_rlpp}\label{rmpp_upper_bound}
If $\alpha$ is well-ordered and a successor then any $\omega$-model satisfying \SPiti{\alpha+2} also satisfies \RLPP{\alpha}.
\end{theorem}
We have covered the case where $\alpha$ is finite above, and in the case where $\alpha\geq\omega$, we actually only need \SPiti{\alpha+1}.

Fix $\alpha$ and assume $WO(\alpha)$, and fix a model $M$ satisfying \SPiti{\alpha+1}.  The arguments below are carried out in the external model (of \ATR) concerning the internal model.  We write $\delta,\gamma$ for arbitrary elements of $\field(\alpha)$ and $\lambda$ for limits in $\field(\alpha)$.  We also fix a tree $T$ and an ordering $\prec$ belonging to $M$ such that $M\models WF(\prec)$.  We will import as many definitions as possible from Section \ref{arithmetic_rpp}, since for successor levels our definitions are unchanged.

\begin{definition} 
We define $\mathcal{L}^\alpha$ to be the language $\mathcal{L}$ from above, together with, for each limit $\lambda\in\field(\alpha+1)$, a new predicate $V_\lambda$.    For each $\gamma\in\field(\alpha+1)$, we define the \emph{rank $\gamma$ formulas} and the \emph{basic rank $\gamma$ formulas} inductively by:
\begin{itemize} 
  \item $F(i)=j$ where $i,j$ are terms is a basic rank $0$ formula,
  \item All other atomic formulas are rank $0$ formulas,
  \item If $\phi$ is a rank $\gamma$ formula then $\exists x\phi$ and $\forall x\phi$ are basic rank $\gamma+1$ formula,
  \item For any limit $\lambda\in\field(\alpha+1)$ and any $n$, $V_\lambda(n)$ is a basic rank $\lambda$ formula,
  \item The rank $\gamma$ formulas contain the basic rank $\gamma$ formulas and are closed under $\wedge,\vee,\neg,\rightarrow$.
\end{itemize}

We write $\mathcal{F}_\gamma$ for the collection of basic formulas of rank $\gamma$, $\mathcal{F}_{<\gamma}$ for $\bigcup_{\delta<\gamma}\mathcal{F}_\delta$, and write $rk(\phi)$ for the least $\gamma$ such that $\phi$ is a formula of rank $\gamma$.

Fix a G\"odel coding $\lceil\cdot\rceil$ of $\mathcal{L}^\alpha$.  We define $\vdash$ on $\mathcal{L}^\alpha$ by adding two additional clauses to usual deduction relation for first-order logic:
\begin{itemize} 
  \item $s\vdash V_\lambda(\lceil\phi\rceil)$ iff $rk(\phi)<\lambda$ and $s\vdash\phi$.
  \item If $rk(\phi)\geq\lambda$ then $s\vdash\neg V_\lambda(\lceil\phi\rceil)$.
\end{itemize}
\end{definition}

\begin{definition} 
Let $\alpha$ be a well-ordering.  For each $\gamma\in\field(\alpha+1)$, define $\mathcal{T}_\gamma(T)$ to be the set of consistent, finite sets $s$ of closed basic formulas of rank $\leq \gamma$ such that: 
\begin{itemize} 
  \item If $F(i)=k\in s$ and $i'< i$ then there is a $j'$ such that $F(i')=j'\in s$,
  \item Let $i$ be largest such that for some $j$, the formula $F(i)=j\in s$; then the sequence $\langle F(0),\ldots,F(i)\rangle\in T$,
  \item If $\exists x \phi(x)\in s$ then there is some $i$ such that $s\cap\mathcal{F}_{rk(\phi)}\vdash\phi(i)$,
  \item If $V_\lambda(\lceil\phi\rceil)\in s$ then $rk(\phi)<\lambda$ and $s\cap\mathcal{F}_{\leq rk(\phi)}\vdash\phi$.
\end{itemize}

If $s\in \mathcal{F}_\lambda$, we write $rk(s)=\max\{rk(\phi)\mid \phi\in s\cap\mathcal{F}_{<\lambda}\}$.

We say $s$ decides $V_\lambda(\lceil\phi\rceil)$ if $V_\lambda(\lceil\phi\rceil)\in s$, $V_\lambda(\lceil\neg\phi\rceil)\in s$, or $rk(\phi)\geq\lambda$.

If $\delta\leq \gamma$, define $\pi^\gamma_\delta:\mathcal{T}_\gamma(T)\rightarrow\mathcal{T}_\delta(T)$ by $\pi^\gamma_\delta(s)=\{\phi\in s\mid rk(\phi)\leq \delta\}$.

If $\delta<\lambda$, $t\in\mathcal{T}_\delta(T)$, $s\in\mathcal{T}_\lambda(T)$, we write $t\prec^{+1}s$ if there is a formula $V_\lambda(\lceil\phi\rceil)\in s$ such that $t\vdash\neg\phi$.  
\end{definition}

\begin{definition} 
We extend the definition of $WF'_\gamma\subseteq\mathcal{T}_\gamma(T)$ by adding a definition for limits:
\begin{itemize} 
  \item $WF'_\lambda(t)$ holds if $WF'_{rk(t)}(\pi^\lambda_{rk(t)}(t))$.
\end{itemize}
\end{definition}
Note that, as above, each $WF_\lambda$ is a $\Pi_\lambda(\Pi^1_1)$ formula.

\begin{lemma}
If $WF_\gamma(s)$ and $s\subseteq t$ then $WF_\gamma(t)$.
\end{lemma}

\begin{lemma}
Let $\delta\leq\gamma\leq\alpha$.  Then:
\begin{enumerate}
  \item If $WF_\delta(\pi^\gamma_\delta(s))$ then $WF_\gamma(s)$, and
  \item If $t\in\mathcal{T}_\gamma(T)\cap\mathcal{T}_\delta(T)$ and $WF_\gamma(t)$ then $WF_\delta(t)$.
\end{enumerate}
\end{lemma}
\begin{proof} 
We prove these by simultaneous induction on $\delta,\gamma$.  (This is necessarily an about $M$ carried out externally to $M$.  Note that the statements of these two parts are of the form $\forall x (\phi(x)\rightarrow\psi(x))$ where $\phi,\psi$ are $\Pi_\gamma(\Pi^1_1)$ or $\Pi_\delta(\Pi^1_1)$ with $\delta\leq\gamma\leq\alpha$, so the statement of the theorem is $\Pi_{\gamma+1}(\Pi^1_1)$, and so at worst $\Pi_{\alpha+1}(\Pi^1_1)$.)  Suppose the claim holds for all pairs $\delta'\leq\gamma'$ with either $\delta'<\delta$ or $\gamma'<\gamma$.

If $\gamma=\delta$, this is immediate.  Suppose $\delta<\gamma$ and $\gamma=\beta+1$.  If $WF_\delta(\pi^{\beta+1}_\delta(s))$ then by IH, $WF_{\beta}(\pi^{\beta+1}_{\beta}(s))$, and therefore by Lemma \ref{wf_increment}, $WF_{\beta+1}(s)$.

If $t\in\mathcal{T}_{\beta+1}(T)\cap\mathcal{T}_\delta(T)=\mathcal{T}_\delta(T)$ and $WF_{\beta+1}(t)$ then $t$ contains no formula of rank $\beta+1$, and so there are no $s\supseteq \pi^{\beta+1}_\beta(t)=t$ such that $s\prec^{+1}t$, so we must have $WF_{\beta}(t)$, and therefore by IH, $WF_\delta(t)$.

Suppose $\gamma=\lambda$ is a limit.  If $WF_\delta(\pi^\lambda_\delta(s))$, we consider two cases.  If $rk(s)\geq\delta$ then we have $WF_{rk(s)}(\pi^\lambda_{rk(s)}(s))$ by applying the first part of IH to $\delta,rk(s)$, and therefore $WF_\lambda(s)$.  If $rk(s)<\delta$ then we have $\pi^\lambda_\delta(s)=\pi^\lambda_{rk(s)}(s)$, so by applying the second part of IH to $rk(s),\delta$, we have $WF_{rk(s)}(s)$.

Suppose $t\in\mathcal{T}_\lambda(T)\cap\mathcal{T}_\delta(T)$ and $WF_\lambda(t)$.  Then since $rk(t)\leq\delta$, we may apply the first part of IH to $rk(t),\delta$ to obtain $WF_\delta(t)$.
\end{proof}

\begin{lemma}
Let $\delta\leq\gamma\leq\alpha$, let $\phi$ be a basic rank $\delta$ formula, let $s\in\mathcal{T}_\gamma(T)$, and suppose that for every $t\supseteq s$ such that $t$ decides $\phi$, $WF_\gamma(t)$.  Then $WF_\gamma(s)$.
\label{decides_ordinal}
\end{lemma}
\begin{proof} 
By main induction on $\delta$ and side induction on $\gamma$.  (Note that for a given $\gamma$ the statement is $\Pi_{\gamma+1}(\Pi^1_1)$, so the statement is $\Pi_{\alpha+1}(\Pi^1_1)$.) 
  The case where $\delta=\gamma=0$ is handled by Lemma \ref{decides_rank_0}.  The case where $\delta=\gamma$ and $\gamma$ is a successor is handled by Lemma \ref{decides_rank_successor}.  The case where $\delta<\gamma$ and $\gamma$ is a successor is handled by Lemma \ref{decide_rank_increment}.

So suppose $\gamma$ is a limit.  It suffices to show that if for every $t\supseteq s$ such that $t$ decides $\phi$, $WF_\gamma(t)$, then $WF'_\gamma(s)$.  If $\delta=\gamma$ then $\phi=V_\gamma(\lceil\psi'\rceil)$; in this case, we set $\psi=\exists x\psi'$ for some variable $x$ not appearing in $\psi$; otherwise $\delta<\gamma$ and we set $\psi=\psi'$.  Set $\beta=\max\{rk(s),rk(\psi)\}<\gamma$.

Suppose $t\supseteq \pi^\gamma_\beta(s)$ and $t$ decides $\psi$.  Set
\[t'=\left\{\begin{array}{ll}
t\cup s&\text{if }\delta<\gamma\\
t\cup s\cup\{V(\lceil\psi'\rceil)\}&\text{if }\delta=\gamma\text{ and }t\vdash\psi\\
t\cup s\cup\{V(\lceil\neg\psi'\rceil)\}&\text{if }\delta=\gamma\text{ and }t\vdash\neg\psi\\
\end{array}\right.\]
Then for any $\rho$, $t'\vdash\rho$ iff $t\vdash\rho$, so $t'$ is consistent.  Also, $t'\supseteq s$ and $t'$ decides $\phi$, so $WF_\gamma(t')$.  Since $rk(t')=\beta$ and $\pi^\gamma_\beta(t')=t$, we have $WF_\beta(t)$.  Since $WF_\beta(t)$ holds for all $t\supseteq\pi^\gamma_\beta(s)$ deciding $\psi$, it follows from IH that $WF_\beta(\pi^\gamma_\beta(s))$ holds.  Therefore $WF_\gamma(s)$ holds. 
\end{proof}

\begin{lemma}
If $WF_\gamma(\emptyset)$ then $WF_0(\emptyset)$.
\label{wf_empty_ordinal}
\end{lemma}
\begin{proof} 
By induction on $\gamma$.  If $\gamma$ is a successor, this follows immediately from IH and Lemma \ref{wf_empty_decrement}.  If $\gamma$ is a limit then since $rk(\emptyset)=0$, we immediately have $WF_\gamma(\emptyset)$.
\end{proof}

The definition of $\widehat{\mathcal{T}_{\gamma+1}}$ given above for $\gamma$ a successor is unchanged.  In particular, we obtain:
\begin{theorem}
Assume $\alpha$ is a successor.  Suppose there is no infinite path (in $M$) through $\widehat{\mathcal{T}_\alpha(T)}^+$ deciding all decisions of rank $\leq\alpha$.  Then $T$ is well-founded (in $M$).
\end{theorem}
\begin{proof} 
We apply Theorem \ref{arithmetic_well_founded}.  The first two assumptions are given by Lemmas \ref{wf_empty_ordinal} and \ref{decides_ordinal}.  Note that Lemma \ref{sigma_2_induction} is identical for $\Pi_{\alpha+1}(\Pi^1_1)$ formulas.
\end{proof}

\begin{definition} 
Let $\phi$ be a closed formula of $\mathcal{L}^\alpha$.  Then $\hat\phi(X,Y)$ is the formula of second-order arithmetic which interprets the function symbol $F$ by $X$, the predicate symbol $\overline{T}$ by $Y$, and $V_\lambda(\lceil\phi\rceil)$ by $\forall Y \left(H_\theta(\lambda,Y)\rightarrow (\lceil\phi\rceil,rk(\phi))\in Y\right)$ for a suitable formula $\theta$.
\end{definition}

\begin{lemma}
Let $\Lambda$ be a path (in $M$) through $\mathcal{T}_\alpha(T)$ deciding all formulas of rank $\leq \alpha$ and let $\sigma_\Lambda$ be the corresponding path through $T$ given by $\sigma_\Lambda(i)=j$ iff $F(i)=j\in\Lambda(m)$ for some (and therefore cofinitely many) $m$.

Then whenever $\phi$ is a closed formula of rank $\leq \alpha$, the following are equivalent:
\begin{enumerate} 
  \item There is an $m$ such that $\phi\in\Lambda(m)$,
  \item $\hat\phi(\sigma_\Lambda,T)$.
\end{enumerate}
\end{lemma}
\begin{proof} 
We proceed by induction on formulas.  The only new case is when $\phi=V(\lceil\psi\rceil)$; this is easily covered by the inductive hypothesis.
\end{proof}

The proof of Theorem \ref{spiti_rmpp} goes through unchanged, showing that the model $M$ satisfying \SPiti{\alpha+1} contains a path satisfying \RLPP{\alpha}, which completes the proof of Theorem \ref{rmpp_upper_bound}.

\bibliographystyle{asl}
\bibliography{AI}
\end{document}